\documentclass{article}
\usepackage[a4paper,top=3cm,bottom=3.5cm,left=3cm,right=3cm,marginparwidth=1.75cm]{geometry}
\usepackage[hidelinks]{hyperref}
\usepackage{amsmath, amssymb, amsthm, yhmath, bbm, graphicx, cleveref, caption, tikz, pgfplots}
\usepackage{capt-of}
\usepackage{subcaption}
\usepackage{comment}

\newcommand{\D}{\partial D}
\renewcommand{\S}{\mathcal{S}}

\newcommand{\outside}{\mathbb{R}^3\setminus \overline{D}}

\renewcommand*{\Re}{\operatorname{Re}}

\newcommand{\de}{\: \mathrm{d}}
\newcommand{\R}{\mathbb{R}}
\newcommand{\iu}{\mathrm{i}\mkern1mu}

\graphicspath{{images/}}
\newcommand{\N}{\mathbb{N}}

\renewcommand{\P}{\mathcal{P}}

\renewcommand{\S}{\mathcal{S}}
\newcommand{\T}{\mathcal{T}}

\newcommand{\p}{\partial}
\renewcommand{\epsilon}{\varepsilon}
\newcommand{\dx}{\: \mathrm{d}}

\renewcommand{\b}[1]{\textbf{#1}}

\newcommand{\ie}{\textit{i.e.}}
\newcommand{\nm}{\noalign{\smallskip}}
\newcommand{\ds}{\displaystyle}

\newcommand{\svdots}{\raisebox{3pt}{\scalebox{.6}{$\vdots$}}}
\newcommand{\sddots}{\raisebox{3pt}{\scalebox{.6}{$\ddots$}}}
\newcommand{\sadots}{\raisebox{3pt}{\scalebox{.6}{$\adots$}}}

\newtheorem{thm}{Theorem}

\newtheorem{lemma}[thm]{Lemma}

\newtheorem{remark}[thm]{Remark}

\numberwithin{equation}{section}
\numberwithin{thm}{section}

\graphicspath{ {figures/} }

\title{High-order exceptional points and enhanced sensing in subwavelength resonator arrays}

\author{Habib Ammari\thanks{\footnotesize Department of Mathematics, ETH Z\"urich, R\"amistrasse 101, CH-8092 Z\"urich, Switzerland (habib.ammari@math.ethz.ch, bryn.davies@sam.math.ethz.ch, erik.orvehed.hiltunen@sam.math.ethz.ch).}\and Bryn Davies\footnotemark[1]  \and Erik Orvehed Hiltunen\footnotemark[1] \and  Hyundae Lee\thanks{\footnotesize  Department of Mathematics, Inha University,  253 Yonghyun-dong Nam-gu,  Incheon 402-751,  Korea (hdlee@inha.ac.kr).} \and Sanghyeon Yu\thanks{\footnotesize Department of Mathematics, Korea University, Seoul 02841, S. Korea (sanghyeon\_yu@korea.ac.kr).}}

\date{}

\begin{document}
	\maketitle
	
\begin{abstract}
	Systems exhibiting degeneracies known as exceptional points have remarkable properties with powerful applications, particularly in sensor design. These degeneracies are formed when eigenstates coincide, and the remarkable effects are exaggerated by increasing the order of the exceptional point (that is, the number of coinciding eigenstates). In this work, we use asymptotic techniques to study $\mathcal{PT}$-symmetric arrays of many subwavelength resonators and search for high-order asymptotic exceptional points. This analysis reveals the range of different configurations that can give rise to high-order asymptotic exceptional points and provides efficient techniques to compute them. We also show how systems exhibiting high-order exceptional points can be used for sensitivity enhancement.
\end{abstract}
	\vspace{0.5cm}
	\noindent{\textbf{Mathematics Subject Classification (MSC2000):} 35J05, 35C20, 35P20.
		
		\vspace{0.2cm}
		
		\vspace{0.5cm}
	
	\section{Introduction}
	
	The behaviour of subwavelength resonators, which are particles interacting strongly with waves at subwavelength scales, is strongly influenced by even very small perturbations. Here, \emph{subwavelength scales} refers to length scales which are significantly smaller than the incident wavelength. Due to this, subwavelength resonators are ideal candidates for building-blocks when designing sensors capable of detecting a variety of phenomena, such as mechanical vibrations, fluctuations in magnetic fields, changes in temperature and the presence of small particles such as viruses and nanoparticles. These devices rely on measuring the shifts in the structure's resonant frequencies, caused by the perturbations \cite{rechtsman2017applied, vollmer2008whispering, vollmer2008single}. A weakness of this approach, however, is that the shift in the resonant frequencies typically scales proportionally to the perturbation, meaning that the shift is very small for small perturbations. This weakness can be overcome through the use of structures with \emph{exceptional points}.
	
	An exceptional point is a point in parameter space at which two or more eigenvalues, and also the corresponding eigenvectors, coincide  \cite{miri2019exceptional, heiss2012physics}. A deep degeneracy of this nature gives rise to structures with remarkable properties. In particular, if a perturbation of order $s$ is made to a structure with an $N$\textsuperscript{th} order exceptional point (\emph{i.e.} one at which $N$ eigenvalues and eigenvectors coincide) then the eigenvalues will generally experience perturbations of order $s^{1/N}$. Thus, when trying to detect small perturbations, the measurable response will be relatively large, offering the grounds for designing enhanced sensing arrays \cite{chen2017exceptional, hodaei2017enhanced, wiersig2016sensors}.
	
	There are different approaches to achieve a non-Hermitian system needed to create exceptional points. One approach is to create a system with unidirectional coupling of the states \cite{wang2019arbitrary}. In systems of high-contrast subwavelength resonators, however, such unidirectional coupling is not achievable. Instead, a non-Hermitian system can be created by introducing gain and loss, corresponding to imaginary material parameters. The assumption of \emph{parity--time ($\P\T$) symmetry} forces the spectrum to be conjugate-symmetric. Exceptional points are then the transition points between a real spectrum and a non-real spectrum which is symmetric around the real axis. High-order exceptional points based on $\P\T$-symmetry have been observed, for example, in \cite{ding2016emergence, ding2015coalescence, yu2020higher, zhang2020high}.
	
	In the setting of subwavelength resonators, the existence and consequences of second-order asymptotic exceptional points was studied in \cite{ammari2020exceptional, ammari2020edge}. Similar structures, also with low-order exceptional points, have been considered in \cite{abdrabou2018exceptional, abdrabou2019exceptional, abdrabou2019formation}. By \emph{asymptotic} exceptional points we mean parameter values such that the eigenvalues and eigenvectors coincide \emph{at leading order} in the asymptotic parameters. In a system with radiative losses (which are inherently asymmetric) this is the best we can hope for. Another difficulty is that in this setting long-range interactions cannot be accurately neglected. In this work, we instead use a fully-coupled approach to study the exceptional points, based on a rigorous discrete approximation to the scattering problem. We both demonstrate analogues of the exceptional points reported in \emph{e.g.} \cite{zhang2020high} and, further, reveal a rich variety of configurations and symmetries which produce high-order exceptional points. Since the systems have open boundaries, radiative losses prevent the systems from exhibiting exact $\P\T$ symmetries. Nevertheless, we will study the asymptotic expansion in the subwavelength regime, where the limiting problem is indeed $\P\T$-symmetric, and demonstrate asymptotic exceptional points.
	
	The structure and main contributions of this work are as follows: we begin, in \Cref{sec:sensing}, by demonstrating the value of high-order exceptional points for sensing applications. We show that if a small particle is introduced into a structure with an  $N$\textsuperscript{th} order exceptional point, then one of the eigenvalues will experience a perturbation that is of the same order as the $N$\textsuperscript{th} root of the small particle's volume. The remaining sections are devoted to the study of high-contrast  subwavelength resonators. The \emph{capacitance matrix approximation} is presented in \Cref{sec:prelim}, which provides a rigorous, $\P\T$-symmetric approximation to the differential problem. The existence of third-order asymptotic exceptional points is proved in \Cref{sec:third}. When the order $N$ is higher than three, the study of exceptional points reduces to the study of a system of $N$ polynomial equations of order $N$. Analytical solutions of these systems are beyond reach. Instead, we combine asymptotic methods with numerical computations to demonstrate the exceptional points. For $N=4$, we find four solutions, with striking symmetries of the gain/loss distribution. These solutions continue to higher orders, and, moreover, the number of distinct solutions rapidly increase as $N$ increases. Finally, in \Cref{sec:sensing_position}, we return to the original motivation for high-order exceptional points, and numerically demonstrate the enhanced sensing in a system of subwavelength resonators.

	\section{Implications for enhanced sensing} \label{sec:sensing}
	
	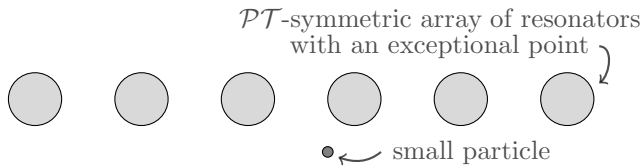
\begin{figure}[h]
		\centering
		\begin{tikzpicture}[scale=0.7]
		\foreach \x in {0,2,...,10} {\draw[fill=gray!30!white] (\x,0) circle (0.5);}
		\draw[fill=gray] (5.5,-1) circle (0.1);
		\def\opac{0.7};
		\draw[->,thick,opacity=\opac] (10.6,1) to [out=0,in=40] (10.6,0.3);
		\node[opacity=\opac] at (7.6,1.5) {$\mathcal{PT}$-symmetric array of resonators};
		\node[opacity=\opac] at (7.6,1) {with an exceptional point};
		\draw[->,thick,opacity=\opac] (6.5,-1) to [out=-140,in=-20] (5.7,-1.1);
		\node[opacity=\opac] at (8.2,-1) {small particle};
		\end{tikzpicture}
		\caption{A system with an exceptional point will experience enhanced eigenfrequency splitting in response to small perturbations, for example due to the introduction of a small particle.}
	\end{figure}

	In order to motivate our forthcoming search for high-order exceptional points in systems of subwavelength resonators, we first explore the use of such a system in enhanced sensing applications. In particular, we wish to understand the behaviour of a system with an $N$\textsuperscript{th} order exceptional point in response to the introduction of a small particle. Recall that the outgoing Helmholtz Green's function $G^k$ is given by
	$$
	G^k(x,y) := -\frac{e^{\iu k|x-y|}}{4\pi|x-y|}, \quad x,y \in \R^3, x\neq y, k\geq 0,
	$$
	where ``outgoing'' is taken to mean that it satisfies the Sommerfeld radiation condition. For a function $u$ with wave number $k$, this condition is given by
	\begin{equation}\label{eq:sommerfeld}
	\lim_{|x|\rightarrow \infty}|x| \left(\frac{\p}{\p |x|} -\iu k\right)u(x) = 0.
	\end{equation}
	Consider a general setting where time-harmonic waves propagate through a material with material parameters described by a function $m\in L^\infty(\mathbb{R}^3)$, which we assume is constant outside of a compact set. In view of the Jordan-type decompositions established in \cite{ammari2015super}, we suppose that the system has an $N$\textsuperscript{th} order exceptional point in the sense that the corresponding Green's function, which is defined as the solution to the Helmoltz problem
	\begin{equation}
	\left\{
	\begin{array} {ll}
	\left(\Delta_x + m(x)k^2\right) G_m^k(x,y)  = \delta_y(x) \qquad \text{in } \R^3, \\[0.3em]
	G_m^k(x,y) \qquad \text{satisfies the Sommerfeld radiation condition as }  |x| \rightarrow \infty,
	\end{array}
	\right.
	\end{equation}
	has the form
	\begin{equation} \label{eq:app_Gexcep}
	G_m^k(x,y) = G^k(x,y)+ \sum_{j=1}^N \frac{\varphi_j(x)\varphi_j(y)}{(k^2-(k^*)^2)^j}
	+ R(k,x,y),
	\end{equation}
	in a neighbourhood of $k^*$, where $k^*\in\mathbb{C}$ is the single $N$\textsuperscript{th} order subwavelength resonant frequency of the system, $\{\varphi_1,\dots,\varphi_N\}\subset H_{\mathrm{loc}}^1(\mathbb{R}^3)$ are generalized eigenmodes associated to $k^*$ and the remainder $R$ is a holomorphic function of $k$ that is smooth as a function of $x$ and $y$. As usual, for a set $A\subset \R^3$, $H^1(A)$ denotes the Sobolev space consisting of square-integrable functions whose weak derivatives are square-integrable, while $H_{\mathrm{loc}}^1(A)$ denotes the subset of $H^1(A)$ whose functions, and weak derivatives thereof, are locally square-integrable.


	Suppose that a small particle $\Omega$ is introduced to the system, which is small in the sense that $\Omega=z+s^{1/3}B$ for some fixed domain $B$, fixed centre $z\in\mathbb{R}^3$ and small size $0<s\ll1$. We have that the volume of $\Omega$ satisfies $|\Omega|=O(s)$ as $s\to0$. By the asymptotic Gohberg-Sigal theory developed in \cite{ammari2009layer}, the perturbed problem will have $N$ resonant modes with frequencies in a neighbourhood of $k^*$, as $s\to0$. A resonant mode of the new system will satisfy the problem
	\begin{equation} \label{eq:small_particle_scattering}
	\left\{
	\begin{array} {ll}
	(\Delta + m(x)k^2) u(x)  = 0 \qquad \text{in } \R^3 \setminus {\Omega}, \\[0.3em]
	(\Delta + \tau(x) k^2) u(x)  = 0 \qquad \ \text{in } \Omega, \\[0.3em]
	u(x) \qquad \text{satisfies the Sommerfeld radiation condition as }  |x| \rightarrow \infty,
	\end{array}
	\right.
	\end{equation}
	where $\tau\in L^\infty(\Omega)$ describes the material parameters within $\Omega$ and we assume that $\tau\not\equiv m$. 
	Thanks to the Lippmann-Schwinger representation \cite{ammari2018mathematical}, we have that
	\begin{equation*}
	u(y)=k^2\int_\Omega (m(x)-\tau(x))G_m^k(x,y)u(x)\de x,
	\end{equation*}
	which, using the decomposition \eqref{eq:app_Gexcep}, becomes
	\begin{equation*} 
	\begin{split}
	u(y)&= k^2\int_\Omega (m(x)-\tau(x))G^k(x,y)u(x)\de x \\&\quad+
	k^2\sum_{j=1}^{N}\frac{\varphi_j(y)}{(k^2-(k^*)^2)^j}\int_\Omega (m(x)-\tau(x))\varphi_j(x)u(x)\de x \\&\quad+
	k^2\int_\Omega (m(x)-\tau(x))R(k,x,y)u(x)\de x.
	\end{split}
	\end{equation*}
	Define the operator $T:L^2(\Omega)\to L^2(\Omega)$ as
	\begin{equation*} \label{defn:T}
	T[v](y)=v(y)-k^2\int_\Omega (m(x)-\tau(x))G^k(x,y)v(x)\de x.
	\end{equation*}
	For sufficiently small $s$, $T$ is invertible (\emph{cf.} \cite[Lemma~3.2]{ammari2015super}), so we may write that
	\begin{equation*}
	u(y)=k^2\sum_{j=1}^{N}\frac{T^{-1}[\varphi_j(y)]}{(k^2-(k^*)^2)^j}\int_\Omega (m(x)-\tau(x))\varphi_j(x)u(x)\de x + \tilde{R}^k[u](y),
	\end{equation*}
	where $\tilde{R}^k$ is holomorphic as a function of $k$ (in a neighbourhood of $k^*$) and its operator norm satisfies $\|\tilde{R}^k\|_{\mathcal{L}(L^2(\Omega),L^2(\Omega))}\to0$ as $s\to0$ for each fixed $k$ \cite{ammari2009layer}. Multiplying by $(m-\tau)\varphi_l$ for $l=1,\dots,N$ and integrating over $\Omega$ yields an approximate matrix eigenvalue problem
	\begin{equation}
	v = k^2Av + r,
	\end{equation}
	where $v\in\mathbb{C}^N$, $A\in\mathbb{C}^{N\times N}$ and $r\in\mathbb{C}^N$ are given by
	\begin{align*}
	v_j &= \int_\Omega (m(x)-\tau(x))\varphi_j(x)u(x)\de x, \\
	A_{ij}&= \frac{1}{(k^2-(k^*)^2)^j} \int_\Omega (m(x)-\tau(x))\varphi_i(x)T^{-1}[\varphi_j](x)\de x, \\
	r_j &= \int_\Omega (m(x)-\tau(x))\varphi_j(x) \tilde{R}[u](k,x) \de x.
	\end{align*}
	Since $r=O(s)$ as $s\to0$, the resonant modes of the perturbed system will approximately satisfy, up to leading order in $s$, the problem
	\begin{equation} \label{eq:sense_problem}
	\det\left(I-k^2A\right)=0.
	\end{equation}
	
	Gohberg-Sigal theory tells us that as $s\to0$, there will be a small perturbation of the original eigenvalue $k^*$, so we write $k=k^*+\nu$ for some small $\nu$. Then, we may write that $I-k^2A=\nu^{-N}(B+C)$ where $C$ is a matrix which has norm $\|C\|_\infty=O(\nu)$ and (vanishing values suppressed)
	\begin{equation*}
	B=\begin{pmatrix}
	\nu^N-\nu^{N-1}\lambda_1 & & & \eta_1 \\
	& \ddots & & \vdots \\
	& & \nu^N-\nu \lambda_{N-1} & \eta_{N-1}\\
	& & & \nu^N-\lambda_N
	\end{pmatrix},
	\end{equation*}
	with the constants $\lambda_1,\dots,\lambda_N,\eta_1,\dots,\eta_{N-1}$ given by $\lambda_i= \int_\Omega (m-\tau)\varphi_iT^{-1}[\varphi_i]\de x$ and $\eta_i= \int_\Omega (m-\tau)\varphi_iT^{-1}[\varphi_N]\de x$. Since $\|C\|_\infty$ is small the Generalized Rouch\'e Theorem, as given in \cite{ammari2009layer}, can be used to determine the singularities of $B+C$. In particular, there will be a one-to-one correspondence (up to multiplicity) between the singularities of $B+C$ and of $B$, and these singularities will be asymptotically close as $s\to0$. Expanding $\det(B)$, we see that it has singularities satisfying $\nu=(\lambda_j)^{1/j}$ for $j=1,\dots,N$. In particular, $j=N$ corresponds to a resonant mode with frequency $k$ which satisfies, in an asymptotic sense,
	\begin{equation*}
	k-k^*\approx\left( \int_\Omega (m-\tau)\varphi_NT^{-1}[\varphi_N]\de x\right)^{1/N}.
	\end{equation*}
	Approximating this integral as $s\to0$, we see that
	\begin{equation} \label{eq:perturbation}
	k-k^*\approx\left( \eta_z|\Omega| \right)^{1/N},
	\end{equation}
	where $|\Omega|$ is the volume of $\Omega$ and $\eta_z=(m(z)-\tau(z))\varphi_N(z)T^{-1}[\varphi_N](z)$.
	
	\begin{remark}
		It is clear that in order for \eqref{eq:perturbation} to offer a useful approach to enhance the shift in the resonant frequency the constant $\eta_z$ should be maximised by carefully positioning the small particle.  We will return to this point in \Cref{sec:sensing_position}, where we will examine how $\eta_z$ varies as a function of the particle's position, for several of the high-order exceptional points which we find below.
	\end{remark}

	\begin{remark}
		In this expository section, we considered a relatively simple setting \eqref{eq:small_particle_scattering} in order to streamline the analysis. This argument could be easily generalized to other settings, such as those considered below.
	\end{remark}
	
	\section{Subwavelength resonators} \label{sec:prelim}
	In this section, we set out the subwavelength resonance problem that will be studied in the remainder of this work. We also introduce the capacitance matrix formulation and the dilute approximation that will form the basis of our search for high-order asymptotic exceptional points.
	\subsection{Problem description}	
	
	We will study a structure composed of $N$ resonators $D_1, D_2,..., D_N $ which are pairwise disjoint subsets of $\R^3$ such that $\p D_i \in C^{1,s}$ for $0 < s < 1$.	In our search for high-order exceptional points, we will restrict ourselves to the case where the resonators are all of equal volume. We assume that the material inside the $i$\textsuperscript{th} resonator $D_i$ has complex-valued material parameters $\kappa_i \in \mathbb{C}$ and $\rho_i \in \mathbb{C}$. The corresponding parameters $\kappa, \rho$ of the surrounding material are assumed to be real. We denote the frequency of the waves by $\omega$ and define the parameters, for $i=1,...,N$,
	$$v_i = \sqrt{\frac{\kappa_i}{\rho_i}}, \quad v = \sqrt{\frac{\kappa}{\rho}}, \quad \delta_i = \frac{\rho_i}{\rho}, \quad k = \frac{\omega}{v}, \quad k_i = \frac{\omega}{v_i}.$$	
	In the frequency domain, the time-reversal operator $\T$ is given by complex conjugation, while the parity operator $\P$ is given by $\mathcal{P}: \R^3 \rightarrow \R^3 $,
	$$\mathcal{P}(x) = -x.$$
	We assume that the collection of resonators $D= \cup_{i=1}^N D_i$ is $\P\T$-symmetric, which means that
	$$\mathcal{P} D = D$$
	and that, for indices $i$ and $j$ such that $\P D_i = D_j$, it must hold that
	$\kappa_i = \overline{\kappa_j}$ and  $\rho_i = \overline{\rho_j}$.
	The imaginary parts can be interpreted as the magnitude of the gain or loss. With these assumptions we define the material contrast $\delta:=|\delta_1|$, which will be our asymptotic parameter. We will assume that
	$$ \delta \ll 1,$$
	and that $\delta_i=O(\delta)$ for all $i>1$ while $v_i = O(1)$ for all $i=1,...,N$. We study the wave resonance problem
	\begin{equation} \label{eq:scattering}
	\left\{
	\begin{array} {ll}
	\ds \Delta {u}+ k^2 {u}  = 0 & \text{in } \R^3 \setminus D, \\[0.3em]
	\ds \Delta {u}+ k_i^2 {u}  = 0 & \text{in } D_i, \ i=1,...N, \\
	\nm
	\ds  {u}|_{+} -{u}|_{-}  = 0  & \text{on } \partial D, \\
	\nm
	\ds  \delta_i \frac{\partial {u}}{\partial \nu} \bigg|_{+} - \frac{\partial {u}}{\partial \nu} \bigg|_{-} = 0 & \text{on } \partial D_i, \ i=1,...N, \\
	\nm
	\ds u(x) \quad \ \text{satisfies the}&\hspace{-7pt}\text{Sommerfeld radiation condition as }  |x| \rightarrow \infty.
	\end{array}
	\right.
	\end{equation}
	Here, $|_+$ and $|_-$ denote the limits from the outside and inside of $D$, respectively. Recall that the Sommerfeld radiation condition is specified in  \eqref{eq:sommerfeld}. We will study the solutions of \eqref{eq:scattering} by rigorously decomposing them in terms of their subwavelength resonant modes. We say that a frequency $\omega$ is a \emph{resonant frequency} if the real part of $\omega$ is positive and there is a non-zero solution $u$ (which is known as the \emph{resonant mode} associated with $\omega$) to the problem \eqref{eq:scattering}. Moreover, we say that a resonant frequency $\omega$ is a \emph{subwavelength resonant frequency} if $\omega\to 0$ as $\delta \to 0$.

	\subsection{Capacitance-matrix analysis} \label{sec:cap}
	Our approach to solving \eqref{eq:scattering} in the case that $u^{in}=0$ is to study the \emph{(weighted) capacitance matrix}. We will see that the eigenstates of this $N\times N$-matrix characterize, at leading order in $\delta$, the resonant modes of the system. This approach offers a rigorous discrete approximation to the differential problem \eqref{eq:scattering}.
	
	Let $\S_D^k$ be the single layer potential, defined by
	\begin{equation*} \label{eq:Sdef}
	\S_D^k[\phi](x) := \int_{\partial D} G^k(x,y)\phi(y) \dx \sigma(y), \quad x \in \R^3.
	\end{equation*}
	We will use the notation $\S_D$ for $\S_D^0$, \emph{i.e.} for the Laplace single layer potential. Since we are working in three dimensions, the Laplace single layer potential is known to be invertible as a map from $L^2(\D)$ to $H^1(\D)$. Further properties of the single layer potential can be found in \emph{e.g.} \cite{ammari2018mathematical}.
	
	In order to introduce the notion of capacitance, we define the functions $\psi_j$, for $j=1,...,N$, as 
	\begin{equation*}
	\psi_j=\S_D^{-1}[\chi_{\p D_j}],
	\end{equation*}
	where $\chi_A:\mathbb{R}^3\to\{0,1\}$ is used to denote the characteristic function of a set $A\subset\mathbb{R}^3$.	The capacitance coefficients $C_{ij}$, for $i,j=1,...,N$, are then defined as
	\begin{equation*}
	C_{ij}=-\int_{\D_i} \psi_j\de\sigma.
	\end{equation*}
	The matrix $C = (C_{ij})$, for $i,j=1,...,N$, is called the \emph{capacitance matrix}. We will define $a\in\mathbb{R}$ to be such that $\Re (v_1^2\delta_1) = \delta a$ and assume that $ a \neq 0$. We then define the weight matrix $V = (V_{ij})$ to be the diagonal matrix with non-zero entries given by
	\begin{equation} \label{eq:V}
	V_{ii} = \frac{v_i^2\delta_i}{\delta a}, \quad i = 1,...,N.
	\end{equation}
	We will see that the factor $a$ just corresponds to a rescaling of all the subwavelength resonant frequencies, and the factor $\delta$ implies that the entries of $V$ scale as $O(1)$ for small $\delta$. Finally, we define the \emph{weighted capacitance matrix} $C^v$ as
	\begin{equation} \label{eq:cap_matrix}
	C^v:= \delta aVC = \left(\begin{smallmatrix}
	 v_1^2\delta_1  {C_{11}} & v_1^2\delta_1 {C_{12}}  & \cdots & v_1^2 \delta_1 {C_{1N}}  \\
	{v_2^2\delta_2}C_{21} & {v_2^2\delta_2}C_{22} & \cdots & v_2^2\delta_2 C_{2N} \\
	\svdots & \svdots & \sddots & \svdots\\
	{v_N^2\delta_N}C_{N1} & {v_N^2\delta_N}C_{N2} & \cdots & v_N^2\delta_3 C_{NN}
	\end{smallmatrix}\right).
	\end{equation}
	This has been weighted to account for the different material parameters inside the different resonators, see \emph{e.g.} \cite{ammari2020close, ammari2017double} for other variants in slightly different settings, such as when the resonators have different volumes.


	We define the functions $S_i^\omega$ as	
	$$S_i^\omega(x) = \begin{cases}
	\S_{D}^{k}[\psi_i](x), & x\in\outside,\\
	\S_{D}^{k_j}[\psi_i](x), & x\in D_j, \ j=1,...,N.\\
	\end{cases}
	$$
	To simplify the notation, we also define the vector of functions $\underline{S}^{\omega}$ as 
	$$\underline{S}^{\omega}(x) = \begin{pmatrix} S_1^\omega(x) \\
	\vdots \\
	S_N^\omega(x)
	\end{pmatrix}.$$
	Then, for a vector $\underline{q} = (q_1,...,q_N)^\mathrm{T} \in \mathbb{C}^N$, where $^\mathrm{T}$ denotes the transpose, we write $\underline{q} \cdot \underline{S}^{\omega}$ to denote the dot product
	$$\underline{q} \cdot \underline{S}^{\omega}(x) = \sum_{i=1}^Nq_i S_i^\omega(x).$$
	The following theorem is a straightforward generalization of \cite[Lemma 2.1, Theorem 2.2]{ammari2020exceptional}.
	\begin{thm} \label{thm:res}
		Let $(\lambda_i, \underline{q_i})$ be the eigenpairs of the weighted capacitance matrix $C^v$. As $\delta \rightarrow 0$, the subwavelength resonant frequencies $\omega_i$ satisfy the asymptotic formula
		$$\omega_i = \sqrt{\frac{\lambda_i}{|D_1|}} + O(\delta), \quad i = 1,\dots,N,$$
		where $|D_1|$ is the volume of each resonator and the branch of the square root is chosen with positive real part. Moreover, the corresponding resonant modes $u_i$ satisfy the asymptotic formula
		$$u_i(x) = \underline{q_i} \cdot \underline{S}^{\omega}(x) + O(\delta^{1/2}), \quad i = 1,\dots,N.$$
	\end{thm}	
	
	\subsection{Asymptotic exceptional points in the dilute regime}  \label{sec:dilute}
	The dilute regime corresponds to the limit when the resonator separation becomes large relative to their size. Specifically, we fix the size and shape of each resonator and assume that the separation scales in proportion to $\epsilon^{-1}$. We wish to study the behaviour of the system as $\epsilon\rightarrow 0$. The following lemma was proved in \cite{ammari2020topological} (up to modification by rescaling).
	\begin{lemma} \label{lem:dilute}
		Consider a dilute system of $N$ identical subwavelength resonators, given by
		$$
		D = \bigcup_{j=1}^N ( B + \epsilon^{-1}z_j),
		$$
		where $0<\epsilon\ll 1$, $B$ is some fixed domain and $\epsilon^{-1}z_j$ represents the position of each resonator. Here, $z_j$ and the size of $B$ are of order one. In the limit $\epsilon\rightarrow 0$,  the asymptotic behaviour of the capacitance coefficients are given by
		$$
		C_{ij} = \begin{cases}
		\mathrm{Cap}_B + O(\epsilon^2), &\quad i=j,
		\\
		\ds -\epsilon \frac{(\mathrm{Cap}_B)^2}{4\pi |z_i-z_j|} + O(\epsilon^2), &\quad i\neq j,
		\end{cases}
		$$
		where $\mathrm{Cap}_B:=-\int_{\partial B} \S_B^{-1}[\chi_{\partial B}]\de\sigma$.
	\end{lemma}
	We will assume that $D_1,\dots,D_N$ are all given by translations of some domain $B$ which is parity-symmetric (\ie{} it satisfies $\P B = B$). For simplicity, we will fix the size of $B$ to be such that $\mathrm{Cap}_B=4\pi$ (this holds \emph{e.g.} if $B$ is the unit sphere). We let $\epsilon>0$ be a small parameter and define the resonators as
	$$
	D_i = B - \left(i-\frac{N+1}{2}\right)(\epsilon^{-1},0,0).
	$$
Notice that the resonators are equally spaced along the $x$-axis and that all the resonators are far away from each other when $\epsilon$ is small. Then, using the matrix $V$ as defined in \eqref{eq:V}, we define the matrix $C_d^v$ as
\begin{equation}
\label{eq:Cdv}C_d^v= V\left(\begin{smallmatrix}
1 & -\epsilon & -\epsilon/2 & \cdots & -\epsilon/(N-1)	\\
-\epsilon & 1 & - \epsilon & \cdots & -\epsilon/(N-2)	\\
-\epsilon/2 & -\epsilon & 1 & \cdots & -\epsilon/(N-3) \\
\svdots & \svdots & \svdots & \sddots & \svdots \\
-\epsilon/(N-1) & -\epsilon/(N-2) & -\epsilon/(N-3) & \cdots & 1 \\	
\end{smallmatrix}\right).
\end{equation}
By \Cref{lem:dilute}, we see that $C_d^v$ gives a dilute approximation of $C^v$ in the sense that
$$C^v = 4\pi a\delta\left(C_d^v + O(\epsilon^2)\right).$$
Also, under this choice of $D$, we can follow the proof of \cite[Theorem 2.2]{ammari2020exceptional} to conclude that the error term in \Cref{thm:res} holds uniformly in $\epsilon$. Therefore, from \Cref{thm:res}, we get the following theorem.
	\begin{thm} \label{thm:res_dilute}
	Let $\gamma_{i}$ be the eigenvalues of $C^v_d$ and $\underline{q_i}$ the corresponding eigenvectors. Then, for small $\epsilon$ and $\delta$, we have the following asymptotic behaviour of $\omega_i$ and $u_i$:
	\begin{align*}
	\omega_i &= \sqrt{\frac{4\pi a\delta \gamma_{i}}{|D_1|}} + O(\delta + \delta^{1/2}\epsilon^2),\\
	u_i(x) &= \underline{q_i} \cdot \underline{S}^{\omega_i} +O(\epsilon^2 + \delta^{1/2}), \quad i=1,...,N.
	\end{align*}
	Here, the error terms hold uniformly for $\epsilon$ and $\delta$ in neighbourhoods of $0$.
\end{thm}
This theorem gives a discrete approximation of the resonant frequencies and eigenmodes, in terms of the eigenvalues and eigenvectors of $C_d^v$. It shows that exceptional points of $C_d^v$ will be asymptotic exceptional points of the full differential equation problem \eqref{eq:scattering}. An $N$\textsuperscript{th} order exceptional point of $C_d^v$ is a set of parameter values such that 
\begin{equation*}
\det (C^v_d - x I) = (\gamma-x)^N\quad \text{and}\quad
\dim \ker (C_d^v - \gamma I) = 1,
\end{equation*}
for some $\gamma$. In what follows, we will study these points by expanding the characteristic polynomial of $C_d^v$ and matching the coefficients to those of $(\gamma-x)^N$.

\section{A third-order exceptional point} \label{sec:third}
	In this section, we consider an array of three resonators and search for a third-order exceptional point. Observe, firstly, that since $N$ is odd the centre resonator must have real material parameters in order to be $\P\T$-symmetric.
	We introduce the notation
	$$v_1^2 \delta_1 := \delta a(1 + \iu{b}), \qquad v_2^2 \delta_2 := \delta a c,  \qquad v_3^2 \delta_3 := \delta a(1 - \iu{b}),$$
	for real-valued parameters $a,b$ and $c$. Notice that $a,b,c=O(1)$. 
	In this case, the matrix $C_d^v$ is given by
$$
C^v_d = \begin{pmatrix}
\ds 1+\iu b   &\ds -(1+\iu b) \epsilon &\ds -(1+\iu b) \epsilon/2 \\[0.3em]
-c \epsilon & c\,  & -c\epsilon
\\[0.3em]
\ds -(1-\iu b)  \epsilon/2 &\ds -(1-\iu b) \epsilon &\ds 1-\iu b 
\end{pmatrix}.
$$
	Next, we shall show that the discrete model matrix $C_d^v$ has an exceptional point of order 3.
	
	The characteristic polynomial of $C_d^v$, which is $\det (C^v_d - x I)$, can be easily computed as
	\begin{align*}
	P(x) &= x^3 -(c+2) x^2 + \left(1+b^2+2c - \frac{\epsilon^2}{4}(1+b^2+8c)\right) x  
	\\
	&\quad -c (1+b^2)  \left(1 - \frac{9}{4}\epsilon^2- \epsilon^3\right).
	\end{align*}
    In order to get an exceptional point of order $3$, we require 
    $$
    P(x) = (x-\gamma)^3 = x^3 -3\gamma x^2 + 3\gamma^2x -\gamma^3,
    $$
    and that $	\dim \ker (C_d^v - \gamma I) = 1$, for some $\gamma$.
	Comparing the above two expressions for $P(x)$, we see that $b,c$ and $\epsilon$ should satisfy
	\begin{align}
		3\gamma &= c+2, \label{eq_zero} \\
	    3\gamma^2 &= 1+b^2+2c - \frac{\epsilon^2}{4}(1+b^2+8c),\label{eq_one}
	    \\
	    \gamma^3 &=c (1+b^2)  \left(1 - \frac{9}{4}\epsilon^2- \epsilon^3\right).\label{eq_two}
	\end{align}	
	
	\begin{lemma} \label{lem:except1}
	For any small $\epsilon>0$, there exist $b>0$ and $c>0$ satisfying \eqref{eq_one} and \eqref{eq_two}.
	Moreover, we have
	\begin{align}
	b &=  b_1 \epsilon + O(\epsilon^2)\label{eq:b_asymp},
	\\
	c &= 1 + c_1\epsilon + O(\epsilon^2),\label{eq:c_asymp}
	\end{align}
	where $b_1$ and $c_1$ are specified as the roots of given polynomials. Therefore, for such $b$ and $c$, the characteristic polynomial of $C_d^v$ is given by
    $$
    P(x) = (x-\gamma)^3 \quad \mbox{with } \gamma = \frac{c+2}{3} = 1 + O(\epsilon).
    $$
	
	\end{lemma}
	\proof
	From \eqref{eq_zero} we have that $\gamma = (c+2)/3$. Then, \eqref{eq_one} can be written as
	\begin{align}
	   1+b^2= \left(1-\frac{\epsilon^2}{4}\right)^{-1} \left[ \frac{(c+2)^2}{3} - 2(1-\epsilon^2) c \right]. \label{eq_one_p}
	\end{align}
	Then, substituting the above into \eqref{eq_two}, we get
	\begin{align}
	     \left(1-\frac{\epsilon^2}{4}\right)\frac{(c+2)^3}{27} = c\left(1 - \frac{9}{4}\epsilon^2- \epsilon^3\right)\left[\frac{(c+2)^2}{3} - 2(1-\epsilon^2) c\right], \label{eq_three}
	\end{align}
	which is a cubic polynomial in $c$. If $\epsilon=0$, it has a solution $c=1$. Since \eqref{eq_three}  is a regularly perturbed cubic equation by small $\epsilon>0$, it has a real root $c$ satisfying 
	$$
	c = 1 + c_1\epsilon + O(\epsilon^2).
	$$
	Substituting the above expansion to \eqref{eq_three}, it is straightforward to see that $c_1$ is the real root of the polynomial $c_1^3+\frac{27}{4}c_1-\frac{27}{8}=0$, \emph{i.e.} $c_1\approx0.483...$ .
	Then, by \eqref{eq_one_p}, we can compute the expansion of $b$ as
	$$
	b =  b_1 \epsilon + O(\epsilon^2),
	$$
	where $b_1=\sqrt{\frac{9}{4}+\frac{c_1^2}{3}}\approx1.53...$ .
	\qed
	
	We next show that, at the exceptional point, all the eigenvectors coalesce.

	\begin{lemma} \label{lem:except2}
	For a given small $\epsilon>0$, let $b$ and $c$ be chosen as in Lemma \ref{lem:except1} and let $\gamma$ be the corresponding eigenvalue of $C_d^v$. 
	Then we have
	$$
	\dim \ker (C_d^v - \gamma I) = 1.
	$$
	
	\end{lemma}
\proof
Since $\gamma = (c+2)/3$, we have
$$
C_d^v - \gamma I = \begin{pmatrix}
1 + \iu b - \frac{c+2}{3} & -(1+\iu b) \epsilon & -(1+\iu b) \epsilon/2
\\
-c\epsilon & c-\frac{c+2}{3} & -c \epsilon
\\
-(1-\iu b) \epsilon/2 & -(1-\iu b) \epsilon & (1-\iu b) - \frac{c+2}{3}
\end{pmatrix}.
$$
We will show that $\mbox{rank}(C_d^v-\gamma I)=2$, which implies the conclusion. 
First of all, since $\gamma$ is an eigenvalue, $\mbox{rank}(C_d^v-\gamma I) < 3$. We therefore only need to show that two of the column vectors of $C_d^v-\gamma I$ are linearly independent. For small $\epsilon>0$, by \eqref{eq:b_asymp} and \eqref{eq:c_asymp}, we have
\begin{equation} \label{eq:C-gI}
C_d^v - \gamma I = \epsilon\begin{pmatrix}
-\frac{c_1}{3}+ \iu b_1 & -1 & -1/2
\\
-1 & \frac{2c_1}{3} & - 1
\\
-1/2 & -1 & -\frac{c_1}{3}- \iu b_1
\end{pmatrix} +O(\epsilon^2).
\end{equation}
Observe that both $b_1$ and $c_1$ are non-zero and real. From the asymptotic behaviour of \eqref{eq:C-gI}, we see that the first and the last column vectors are linearly independent for sufficiently small $\epsilon$. \qed

By the above two lemmas, we see that the trimer has an exceptional point of order $3$.

\begin{thm}
For small $\epsilon$ and $\delta$ the trimer $D$ has an asymptotic exceptional point of order $3$ at the resonant frequency $\omega^*$ satisfying
$$\omega^* = \sqrt{\frac{4\pi a\delta(c+2)}{3|D_1|}} + O(\delta + \delta^{1/2}\epsilon^2).$$
Here, $c \in \R$ is a constant satisfying \eqref{eq_one} and \eqref{eq_two}.
\end{thm}

\begin{remark}
	As discussed in \cite{ammari2020exceptional}, we do not expect the original differential problem \eqref{eq:scattering} to support exact exceptional points, in the sense of exactly degenerate resonant frequencies and coalescence of eigenmodes. The exceptional points studied here are linked to the $\P\T$-symmetry of the problem. Even under the assumption of symmetric gain and loss, the problem \eqref{eq:scattering} is not $\P\T$-symmetric, since the radiation condition swaps sign under complex conjugation. Nevertheless, for small $\epsilon$ and $\delta$, the leading order approximation given by $C_d^v$ is indeed $\P\T$-symmetric, leading to  asymptotic, approximate, exceptional points. This is demonstrated by Figures~\ref{fig:diluteN}~and~\ref{fig:fullN} which depict the coincidence of the eigenvalues of $C_d^v$ and the resonant frequencies of the full differential system, respectively.
\end{remark}

\begin{remark}
	The approach used here (making an approximation under the assumption that the resonators are arbitrarily far apart) can also be used to find the exceptional point supported by a $\mathcal{PT}$-symmetric pair of resonators (see \Cref{sec:2res} for details). This structure was previously studied in \cite{ammari2020exceptional} using a more general approach that requires no assumptions of diluteness.
\end{remark}

\section{Exceptional points of order four} \label{sec:fourth-order}
We now seek fourth-order exceptional points, and assume that $D$ is an array of four $\mathcal{PT}$-symmetric resonators with material parameters given by
$$v_1^2 \delta_1 := \delta a(1 + \iu{b}), \qquad v_2^2 \delta_2 := \delta a (c+\iu d),  \qquad v_3^2 \delta_3 := \delta a(c - \iu{d}),\qquad v_4^2 \delta_4:=\delta a(1 - \iu{b}).$$
In this setting, the matrix $C_d^v$, as defined in \Cref{sec:cap}, is
$$
C^v_d = \begin{pmatrix}
\ds 1+\iu b   &\ds -(1+\iu b) \epsilon &\ds -(1+\iu b) \epsilon/2 & -(1+\iu b) \epsilon/3 \\[0.3em]
-(c+\iu d) \epsilon & c+\iu d\,  & -(c+\iu d) \epsilon & -(c+\iu d) \epsilon/2 \\[0.3em]
-(c-\iu d) \epsilon/2 & -(c-\iu d) \epsilon & c-\iu d\,  & -(c-\iu d) \epsilon \\[0.3em]
-(1-ib)  \epsilon/3 & -(1-ib)  \epsilon/2 & -(1-ib) \epsilon &\ds 1-ib 
\end{pmatrix}
$$
and the characteristic polynomial $P(x)=\det (C^v_d - x I)$ can be computed as
\begin{align*}
P(x) &= x^4 -2(c+1) x^3
\\
&\quad + \left(1+b^2+c(4+c)+d^2 - \epsilon^2\frac{1}{18}(2+2b^2+9c(5+2c)-27bd+18d^2)\right) x^2  
\\
&\quad +\frac{1}{18}\left( 9(c^2+d^2)(-4+9\epsilon^2+4\epsilon^3)+c(1+b^2)(-36+\epsilon^2(49+12\epsilon)) \right)x
\\
&\quad +\frac{1}{144} (1+b^2)(c^2+d^2)\left(144 - \epsilon^2( 520 + 384 \epsilon + 23 \epsilon^2)\right).
\end{align*}
In order to get an exceptional point of order $4$, we require that
$$
P(x) = (x-\gamma)^4 = x^4 -4\gamma x^3 + 6\gamma^2x^2-4\gamma^3 x +\gamma^4,
$$
for some $\gamma$. Comparing the two expressions for $P$, we see that $\gamma = (c+1)/2$ and that
\begin{align}
\frac{3}{2}(c+1)^2 &= 1+b^2+c(4+c)+d^2 - \frac{\epsilon^2}{18}\left(2+2b^2+9c(5+2c)-27bd+18d^2\right), \label{eq:41}
\\
\frac{1}{2}(c+1)^3 &= -\frac{1}{18}\left( 9(c^2+d^2)(-4+9\epsilon^2+4\epsilon^3)+c(1+b^2)(-36+49\epsilon^2+12\epsilon^3) \right), \label{eq:42}
\\
\frac{1}{16}(c+1)^4 &=\frac{1}{144} (1+b^2)(c^2+d^2)(144-\epsilon^2( 520 + 384 \epsilon + 23 \epsilon^2)). \label{eq:43}
\end{align}
We are interested in solutions to this system for small $\epsilon$. At $\epsilon = 0$, we have a unique solution given by
$$ b= 0, \qquad c = 1, \qquad d = 0.$$
For small but nonzero $\epsilon$, we make the ansatz
\begin{align*}
	b =  b_1\epsilon + O(\epsilon^2),  \qquad c = 1 + c_1\epsilon + O(\epsilon^2),  \qquad d =  d_1\epsilon + O(\epsilon^2).
	\end{align*}
	By substituting into \eqref{eq:41}--\eqref{eq:43} we find, after simplifications, that
\begin{align} 
	c_1 ^2 - 2(b_1^2+d_1^2) +\frac{65}{9}=0, \label{eq:b1c1d1.1}\\
	c_1^3 + c_1\left(\frac{49}{9} - 4b_1^2\right) + \frac{16}{3} = 0, \label{eq:b1c1d1.2}\\
	c_1^4 - 16(c_1^2+d_1^2)\left(b_1^2-\frac{1}{9}\right) +\frac{32}{3}c_1 + 16b_1^2-24b_1d_1+\frac{23}{9} = 0. \label{eq:b1c1d1.3}
	\end{align}
	Moreover, it is clear that if $(b_1,c_1,d_1)$ is a solution, then $(-b_1,c_1,-d_1)$ is also a solution. Solving the above system numerically, we obtain 4 solutions up to this symmetry (or 8 solutions in total), presented in \Cref{fig:sols4}.

\begin{figure}
	\begin{subfigure}[b]{0.45\linewidth}
		\centering
		\begin{tikzpicture}[scale=0.5]
		\begin{scope}
		\def\b{1.87}
		\def\d{0.56}
		\draw[fill=gray,gray] (0,0) rectangle (1,\b);
		\node at (0.5,-0.5) {\b};
		\draw[fill=gray,gray] (2,0) rectangle (3,\d);
		\node at (2.5,-0.5) {\d};
		\draw[fill=gray,gray] (4,0) rectangle (5,-\d);
		\draw[fill=gray,gray] (6,0) rectangle (7,-\b);
		\draw[dotted] (-0.5,0) -- (7.5,0);
		\path (-0.5,-2) -- (-0.5,2);
		\end{scope}
		\end{tikzpicture}
		\captionsetup{type=figure}
		\caption{Exceptional point satisfying $b_1d_1 > 0$ and $|b_1| > |d_1|$, with $c_1 = 0.654$.} \label{fig:sol4a}
	\end{subfigure}
	\hspace{20pt}
	\begin{subfigure}[b]{0.45\linewidth}
		\centering
		\begin{tikzpicture}[scale=0.5]
		\begin{scope}
		\def\b{0.0456}
		\def\d{2.00}
		\draw[fill=gray,gray] (0,0) rectangle (1,\b);
		\node at (0.5,-0.5) {\b};
		\draw[fill=gray,gray] (2,0) rectangle (3,\d);
		\node at (2.5,-0.5) {\d};
		\draw[fill=gray,gray] (4,0) rectangle (5,-\d);
		\draw[fill=gray,gray] (6,0) rectangle (7,-\b);
		\draw[dotted] (-0.5,0) -- (7.5,0);
		\path (-0.5,-2) -- (-0.5,2);
		\end{scope}
		\end{tikzpicture}
		\captionsetup{type=figure}
		\caption{Exceptional point satisfying   $b_1d_1 > 0$ and $|b_1| < |d_1|$, with $c_1 = -0.863$.}
	\end{subfigure}
	\vspace{5pt}
	
	\begin{subfigure}[b]{0.45\linewidth}
		\centering
		\begin{tikzpicture}[scale=0.5]
		\begin{scope}
		\def\b{1.70}
		\def\d{-1.13}
		\draw[fill=gray,gray] (0,0) rectangle (1,\b);
		\node at (0.5,-0.5) {\b};
		\draw[fill=gray,gray] (2,0) rectangle (3,\d);
		\node at (2.5,0.5) {\d};
		\draw[fill=gray,gray] (4,0) rectangle (5,-\d);
		\draw[fill=gray,gray] (6,0) rectangle (7,-\b);
		\draw[dotted] (-0.5,0) -- (7.5,0);
		\path (-0.5,-2) -- (-0.5,2);
		\end{scope}
		\end{tikzpicture}
		\captionsetup{type=figure}		
		\caption{Exceptional point satisfying   $b_1d_1 < 0$ and $|b_1| > |d_1|$, with $c_1 = 1.07$.}
	\end{subfigure}
	\hspace{20pt}
	\begin{subfigure}[b]{0.45\linewidth}
		\centering
		\begin{tikzpicture}[scale=0.5]
		\begin{scope}
		\def\b{0.734}
		\def\d{-1.93}
		\draw[fill=gray,gray] (0,0) rectangle (1,\b);
		\node at (0.5,-0.5) {\b};
		\draw[fill=gray,gray] (2,0) rectangle (3,\d);
		\node at (2.5,0.5) {\d};
		\draw[fill=gray,gray] (4,0) rectangle (5,-\d);
		\draw[fill=gray,gray] (6,0) rectangle (7,-\b);
		\draw[dotted] (-0.5,0) -- (7.5,0);
		\path (-0.5,-2) -- (-0.5,2);
		\end{scope}
		\end{tikzpicture}
		\captionsetup{type=figure}
		\caption{Exceptional point satisfying   $b_1d_1 < 0$ and $|b_1| < |d_1|$, with $c_1 = -1.15$.} \label{fig:sol4d}
	\end{subfigure}
	\caption{A system of four $\mathcal{PT}$-symmetric resonators supports four asymptotic exceptional points. Here, we plot the leading order coefficients of the imaginary parts of the material parameters (the gain or loss) at each of the four exceptional points.}
	\label{fig:sols4}
\end{figure}
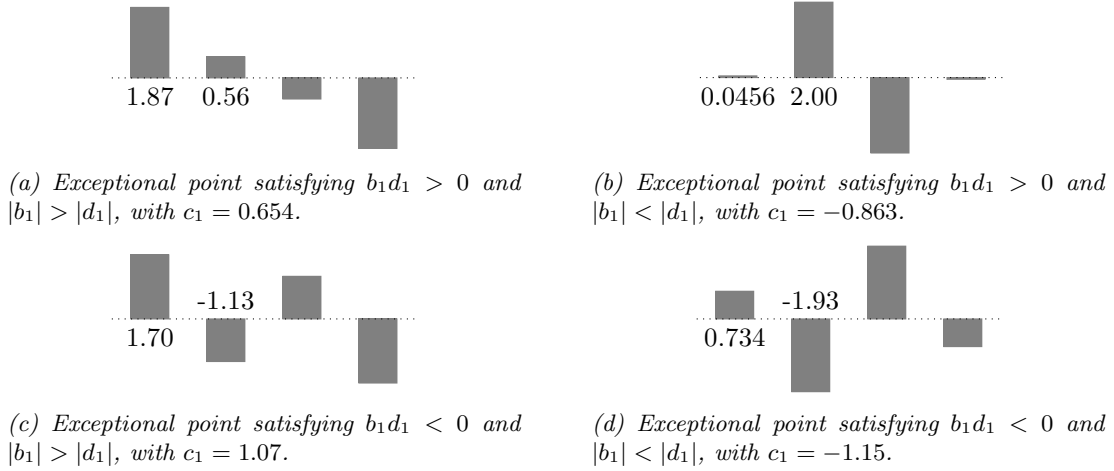

	The four different solutions can be described in terms of the relative magnitude and signs of $b_1$ and $d_1$: each solution corresponds to one of the four cases depending on if $b_1$ and $d_1$ have the same or opposite sign, and if $b_1$ or if $d_1$ is larger in magnitude. This is depicted in \Cref{fig:sols4}. A solution with the qualitative features $b_1d_1 > 0$ and $|b_1| > |d_1|$ was previously observed in the setting of a Hamiltonian system in \cite{zhang2020high}.

We have used formal asymptotics to approximate the continuously differentiable solutions to \eqref{eq:41}--\eqref{eq:43}. Next, we show that, at a solution to this system, all eigenvectors of $C_d^v$ coalesce.
\begin{lemma} \label{lem:except5}
	For a given small $\epsilon>0$, let $b,c$ and $d$ be solutions to \eqref{eq:41}--\eqref{eq:43} and let $\gamma$ be the corresponding eigenvalue of $C_d^v$. 
	Then we have
	$$
	\dim \ker (C_d^v - \gamma I) = 1.
	$$	
\end{lemma}

\begin{proof}
	We will show the equivalent statement that $C_d^v - \gamma I$ has rank $3$. Since $\gamma$ is an eigenvalue of $C_d^v$, the rank is at most $3$. Moreover, since $\gamma = (c+1)/2$, we have for small $\epsilon$,
	$$
	C_d^v - \gamma I = \epsilon\begin{pmatrix}
	-\frac{c_1}{2}+ \iu b_1 & -1 & -1/2 & -1/3\\
	-1 & \frac{c_1}{2} + \iu d_1 & - 1 & -1/2\\
	-1/2 & -1 & \frac{c_1}{2} - \iu d_1 &-1 \\
	-1/3 & -1/2 & -1 & -\frac{c_1}{2} - \iu b_1
	\end{pmatrix} +O(\epsilon^2).
	$$
	The determinant of the $3\times 3$ upper right block is given by
	$$\det \begin{pmatrix}-1 & -1/2 & -1/3\\
	\frac{c_1}{2} + \iu d_1 & - 1 & -1/2\\
	-1 & \frac{c_1}{2} - \iu d_1 &-1
	\end{pmatrix} = -\frac{1}{12}\left((c_1+3)^2+4d_1^2+2\right),$$
	and is negative for any $c_1$ and $d_1$. Therefore, for small $\epsilon$, the rank of $C_d^v - \gamma I$ is at least $3$.
\end{proof}

\section{Exceptional points of arbitrary order} \label{sec:higher}
Here, we study exceptional points in larger systems of resonators. We will consider a $\P\T$-symmetric array with either even or odd number of resonators. In the case of an odd number, analogously to the third order exceptional point, we assume that the centre resonator has no gain or loss. 

To be specific, we consider an array of $N$ resonators with material parameters given by $v_i^2 \delta_i := \delta a(a_i + \iu{b_i})$ for $i=1,...,N.$, for some $a,a_i,b_i\in \R$. We choose $a$ such that $a_1 = 1$. In the case of an even number of resonators, $N = 2n, n \in \N$, we assume that $a_i = a_{2n+1-i}$ and $b_i = -b_{2n+1-i}$, in other words that
$$v_1^2 \delta_1 = \delta a(a_1 + \iu{b_1}), \ \dots, \ v_n^2 \delta_n = \delta a(a_n + \iu{b_n}), \quad v_{n+1}^2 \delta_{n+1} = \delta a(a_n - \iu{b_n}), \ \dots,\ v_{2n}^2 \delta_{2n} = \delta a(a_1 - \iu{b_1}).$$
In the case of an odd number of resonators, $N = 2n + 1, n \in \N$, we assume that $a_i = a_{2n+2-i}$,  $b_i = -b_{2n+2-i}$ and $b_{n+1} = 0$, in other words that
\begin{align*}v_1^2 \delta_1 &= \delta a(a_1 + \iu{b_1}), \ \dots, \ v_n^2 \delta_n = \delta a(a_n + \iu{b_n}), \quad  v_{n+1}^2 \delta_{n+1} = \delta a a_{n+1}, \\ v_{n+2}^2 \delta_{n+2} &= \delta a(a_n - \iu{b_n}), \ \dots,\ v_{2n+1}^2 \delta_{2n+1} = \delta a(a_1 - \iu{b_1}).\end{align*}

In this setting, the dilute capacitance matrix $C^v_d = (C^v_{d,i,j})$, as defined in \Cref{sec:dilute}, is the matrix with entries specified by
$$
C^v_{d,i,j} = \begin{cases}\ds a_i + \iu{b_i}, \quad & i = j, \\[0.3em] \ds -(a_i + \iu{b_i}) \frac{\epsilon}{|i-j|}, & i\neq j.\end{cases}
$$
Again, to have an exceptional point of order $N$ we require that
\begin{align}
\det (C^v_d - x I) &= (\gamma-x)^N, \label{eq:pol} \\
\dim \ker (C_d^v - \gamma I) &= 1, \label{eq:dim}
\end{align}
for some $\gamma$. For general $N$,  the equation \eqref{eq:pol} is a system of $N$ polynomial equations of order $N$, in terms of the $N$ unknown parameters $\gamma, b_1, a_2,b_2,...,a_n,b_n$ and, if $N$ is odd, $a_{n+1}$. For $N$ larger than 4, it is not possible to explicitly derive the solutions, and we will numerically study this system of equations in \Cref{sec:num}. Nevertheless, under the assumption that \eqref{eq:pol} holds, we can derive asymptotic formulas for the unknown parameters analogously to \eqref{eq:b1c1d1.1}--\eqref{eq:b1c1d1.3}.

We begin by observing the following simple equation for $\gamma$:
\begin{equation} \label{eq:gamma}
\gamma = \frac{1}{N}\sum_{i=1}^N a_i.
\end{equation}
As $\epsilon \rightarrow 0$, we also have the following result.
\begin{lemma} \label{lem:exp}
	Assume that there is a solution $a_i, b_i,$ for $i = 1,...,N$,	to \eqref{eq:pol} which is continuous as $\epsilon \rightarrow 0$. Then, we have
	$$a_i = 1 + O(\epsilon), \qquad b_i = O(\epsilon),$$
	and, consequently, $\gamma = 1 + O(\epsilon)$.
\end{lemma}
\begin{proof}
	As $\epsilon\rightarrow 0$, denote the limiting values of $a_i$ and $b_i$ by $a_{i,0}$ and $b_{i,0}$, respectively. Since \eqref{eq:pol} has no terms of order $O(\epsilon)$, we find that
	\begin{align*}
	a_i = a_{i,0} + O(\epsilon), \qquad	b_i = b_{i,0} + O(\epsilon).
	\end{align*}
	We then have 
	\begin{align*}
	\det (C^v_d - x I) &= \prod_{i=1}^N (a_{i,0} + \iu b_{i,0} - x) + O(\epsilon) \\ 
	&=(\gamma_0 - x)^N + O(\epsilon),
	\end{align*}
	where, from \eqref{eq:gamma},
	$$ \gamma_0 = \frac{1}{N}\sum_{i=1}^N a_{i,0}.$$
	Since $\gamma_0 \in \R$, it follows that $b_{i,0} = 0$ for all $i=1,..,N$, and since $a_{1,0} = 1$ it follows that $a_{i,0} = 1$ for all $i=1,...,N$.
\end{proof}
Assume that \eqref{eq:pol} holds for some $\gamma$. Then, for small $\epsilon>0$,  we have by \Cref{lem:exp} that
$$a_i = 1 + a_{i,1}\epsilon + o(\epsilon), \qquad b_i = b_{i,1}\epsilon + o(\epsilon), \qquad \gamma = 1 + \gamma_1\epsilon + o(\epsilon),$$
for some $a_{i,1}, b_{i,1}$ and $\gamma_1$ independent of $\epsilon$. In this case, it holds that
\begin{equation}\label{eq:CmgI}
C_d^v = I + \epsilon C_{d,1}^v + o(\epsilon), \qquad C_{d,1}^v = \left(\begin{smallmatrix}
a_{1,1} + \iu b_{1,1} & -1 & -1/2 & \cdots & -1/(N-1)	\\
-1 & a_{2,1} + \iu b_{2,1} & - 1 & \cdots & -1/(N-2)	\\
-1/2 & -1 & a_{3,1} + \iu b_{3,1} & \cdots & -1/(N-3) \\
\svdots & \svdots & \svdots & \sddots & \svdots \\
-1/(N-1) & -1/(N-2) & -1/(N-3) & \cdots & a_{1,1} - \iu b_{1,1}	
\end{smallmatrix}\right).
\end{equation}
Then
$$C_d^v - \gamma I = \epsilon \left(C_{d,1}^v - \gamma_1I\right) + o(\epsilon).$$
Therefore, for $\gamma$ to be an $N$\textsuperscript{th} order exceptional point of $C_d^v$, we must have that $\gamma_1$ is an $N$\textsuperscript{th} order exceptional point of $C_{d,1}^v$. We can then obtain a system of polynomial equations describing the exceptional point (analogous to \eqref{eq:b1c1d1.1}--\eqref{eq:b1c1d1.3} but for general $N$) by expanding the characteristic polynomial of $C_{d,1}^v$. Moreover, assuming that \eqref{eq:pol} holds for some $\gamma$, a simple way to prove that \eqref{eq:dim} holds is to check that $C_{d,1}^v$ has a one-dimensional kernel.

At an exceptional point $\gamma_1$ of $C_{d,1}^v$, it follows from \Cref{thm:res_dilute} and \eqref{eq:CmgI} that the full system exhibits an asymptotic exceptional point with frequency given by
$$\omega^* = \sqrt{\frac{4\pi a \delta}{|D_1|}(1+\epsilon \gamma_1)} + O(\delta) + \delta^{1/2}o(\epsilon).$$

\section{Numerical computations} \label{sec:num} 
In this section, we perform numerical simulations to illustrate properties and applications of high-order exceptional points. In \Cref{sec:high-num} we numerically demonstrate exceptional points of arbitrary order. In \Cref{sec:sensing_position} we return to the initial question of achieving enhanced sensing using high-order exceptional points, and study the details of how strongly small particle perturbations are enhanced for different particle positions.

\subsection{High-order exceptional points} \label{sec:high-num}
To find the exceptional points of $C_{d,1}^v$, the equation for the characteristic polynomial,
$$\det (C^v_{d,1} - x I) = (\gamma_1-x)^N,$$
was solved numerically (in terms of the unknown parameters $a_{i,1}, a_{i,2}..., b_{i,1}, b_{i,2},... $ and $\gamma_1$), and the solutions satisfying $\dim \ker (C^v_{d,1} - \gamma_1 I) = 1$ were selected. Throughout this section, the computations were performed using spherical resonators with unit radius, $\delta = 1/5000, \epsilon = 0.1$ and $a = 1$. 

\Cref{fig:diluteN} shows, to leading order and for selected $N$, the resonant frequencies as the gain and loss increases from $0$ and crosses the exceptional points. The gain/loss parameters are $b_{i,1} = \tau b_{i,1}^*$, for $0 \leq \tau \leq 2$, where $b_{i,1}^*$ corresponds to the gain/loss at the exceptional point. \Cref{fig:diluteN} shows examples of exceptional points of various orders up to $N=14$. These examples all follow the same pattern, whereby the gain/loss grows linearly away from the centre (previously reported by \cite{zhang2020high}). We expect similar behaviour for even larger $N$, which demonstrates the possibility to create exceptional points of arbitrary order.

\Cref{fig:fullN} shows the resonant frequencies of the full differential system, without making any asymptotic approximations. Here, the resonant frequencies were computed using the multipole method (see \cite[Appendix A]{ammari2020topological} for details). This demonstrates the approximate nature of the exceptional points: due to the radiation condition and the loss of energy to the far field, the frequencies have non-zero imaginary parts even at $\tau=0$. As $\delta \rightarrow 0$, this imaginary part vanishes and the system has an exact exceptional point.
\begin{figure}
	\begin{center}
		\begin{subfigure}[b]{0.45\linewidth}
			\includegraphics[height=5.0cm]{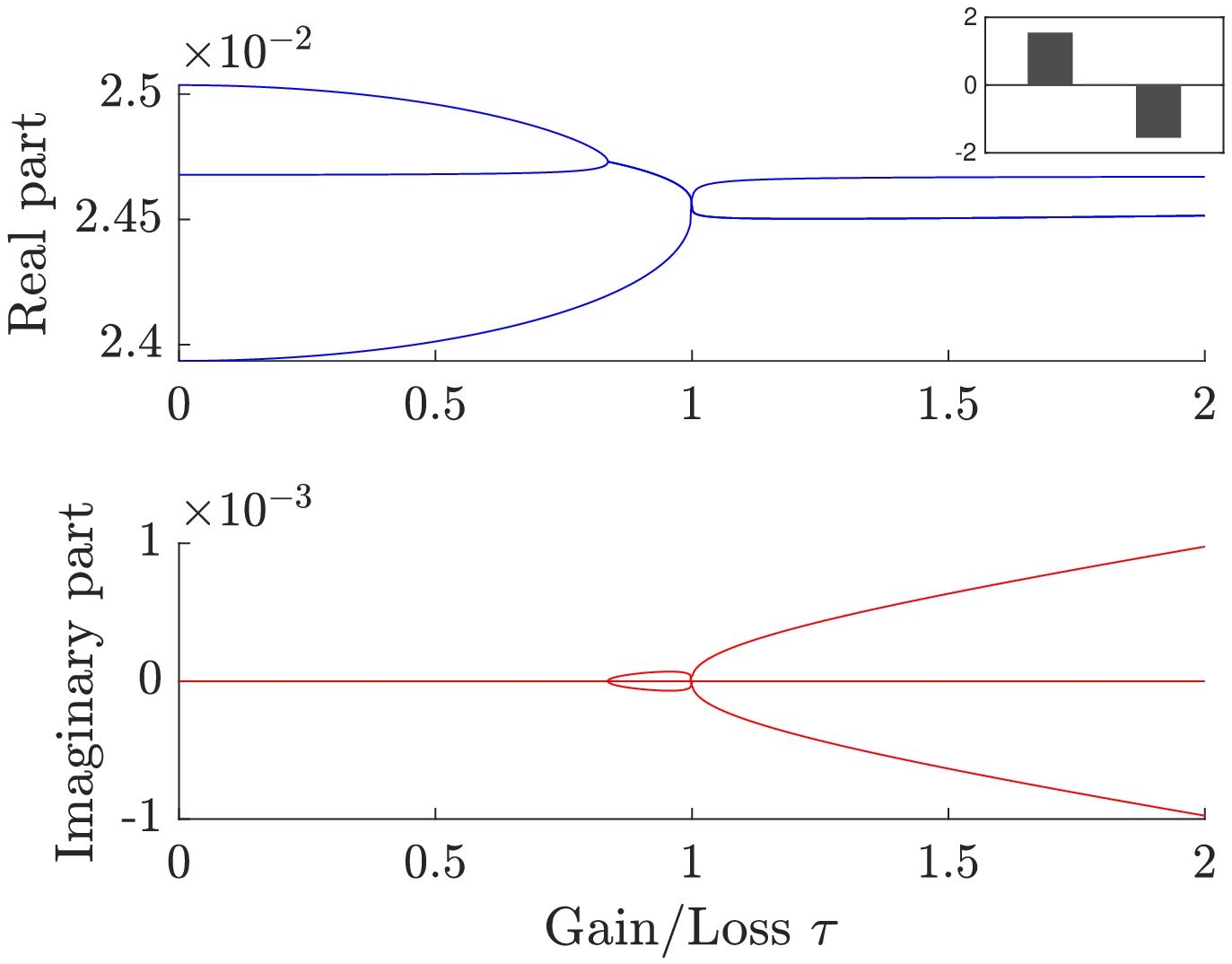}
			\caption{$N=3$} \label{fig:N3}
		\end{subfigure}			
		\hspace{10pt}
		\begin{subfigure}[b]{0.45\linewidth}
			\includegraphics[height=5.0cm]{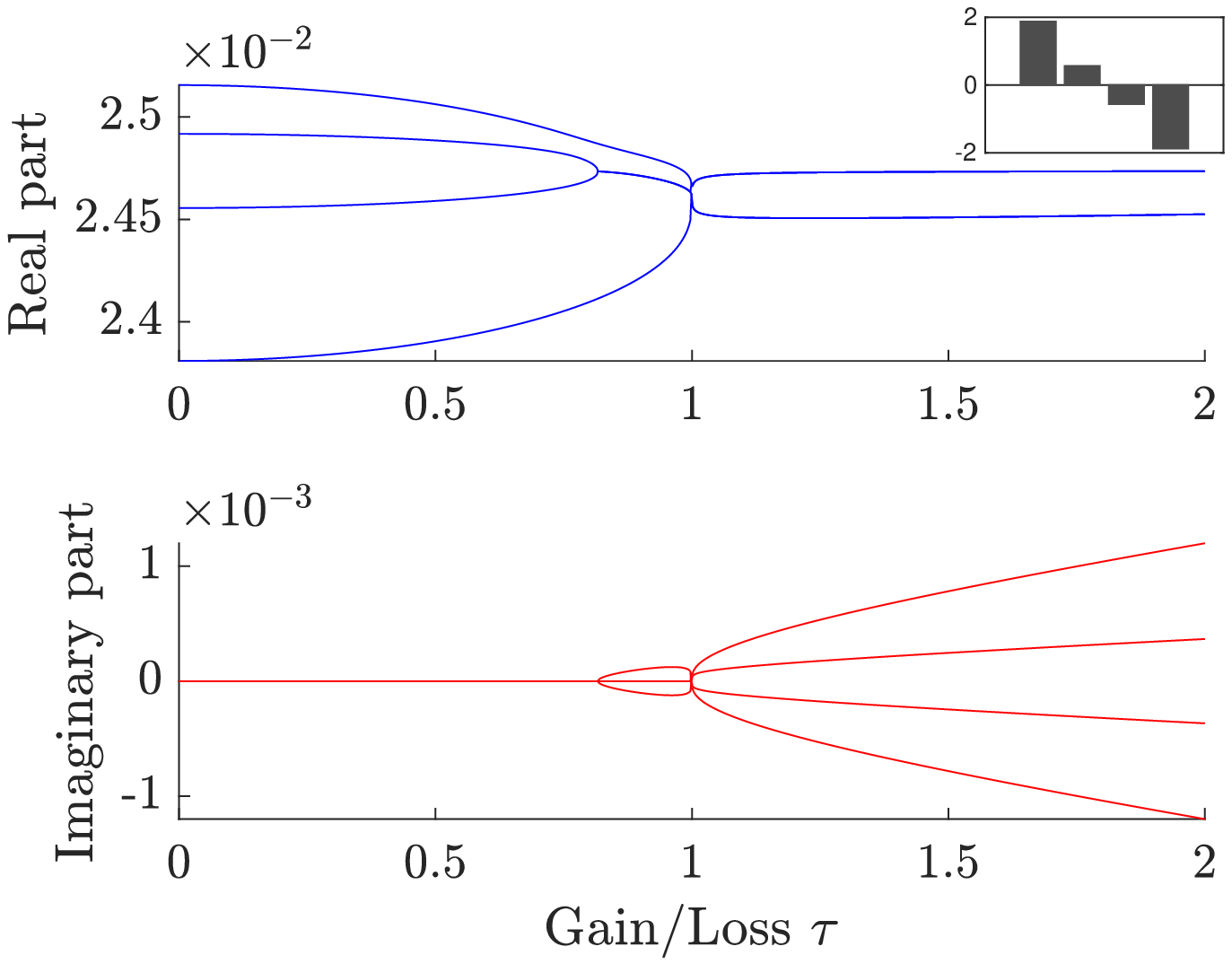}
			\caption{$N=4$.} \label{fig:N4}
		\end{subfigure}	
		\vspace{20pt}
		\begin{subfigure}[b]{0.45\linewidth}
			\includegraphics[height=5.0cm]{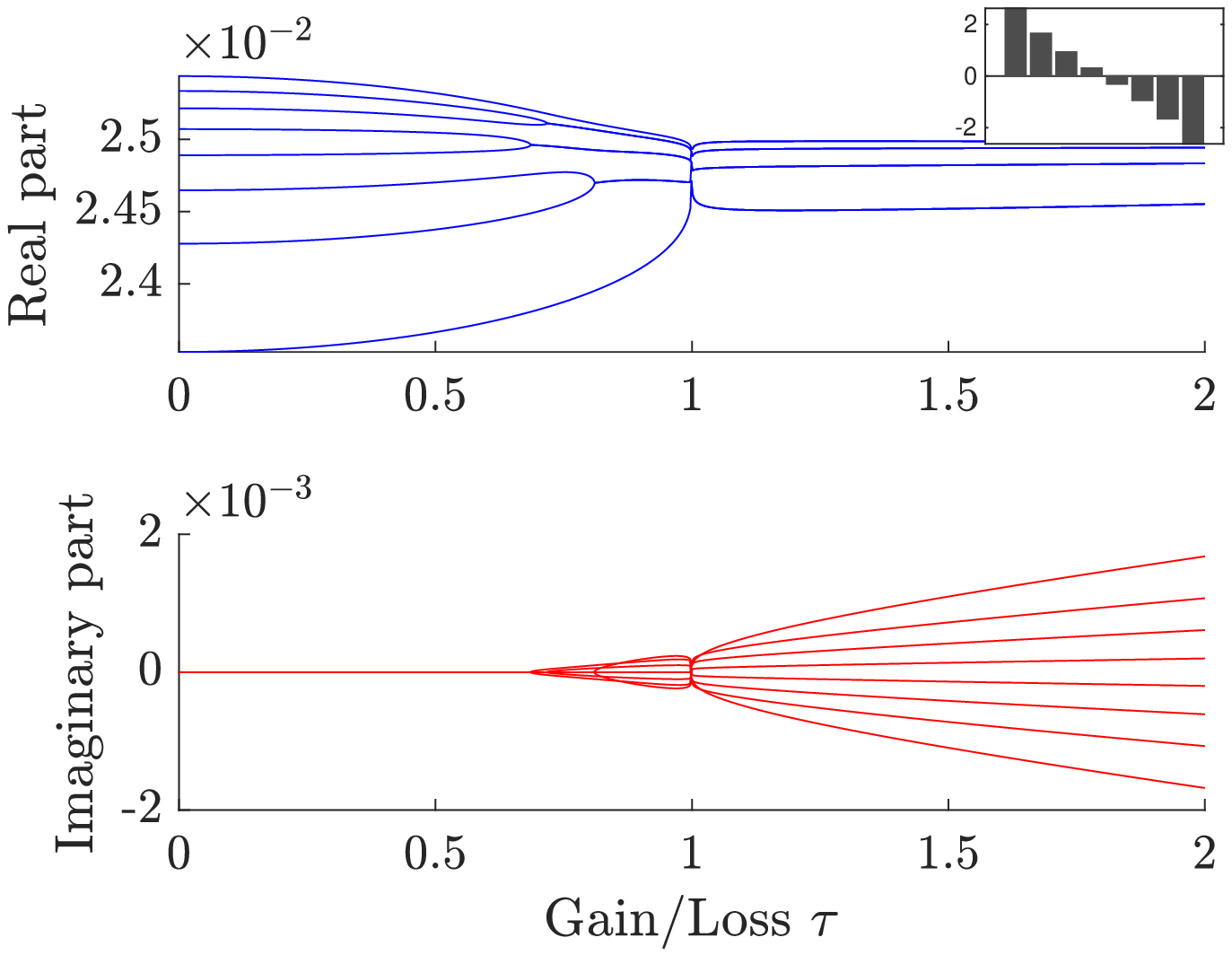}
			\caption{$N=8$.} \label{fig:N5}
		\end{subfigure}			
		\hspace{10pt}
		\begin{subfigure}[b]{0.45\linewidth}
			\includegraphics[height=5.0cm]{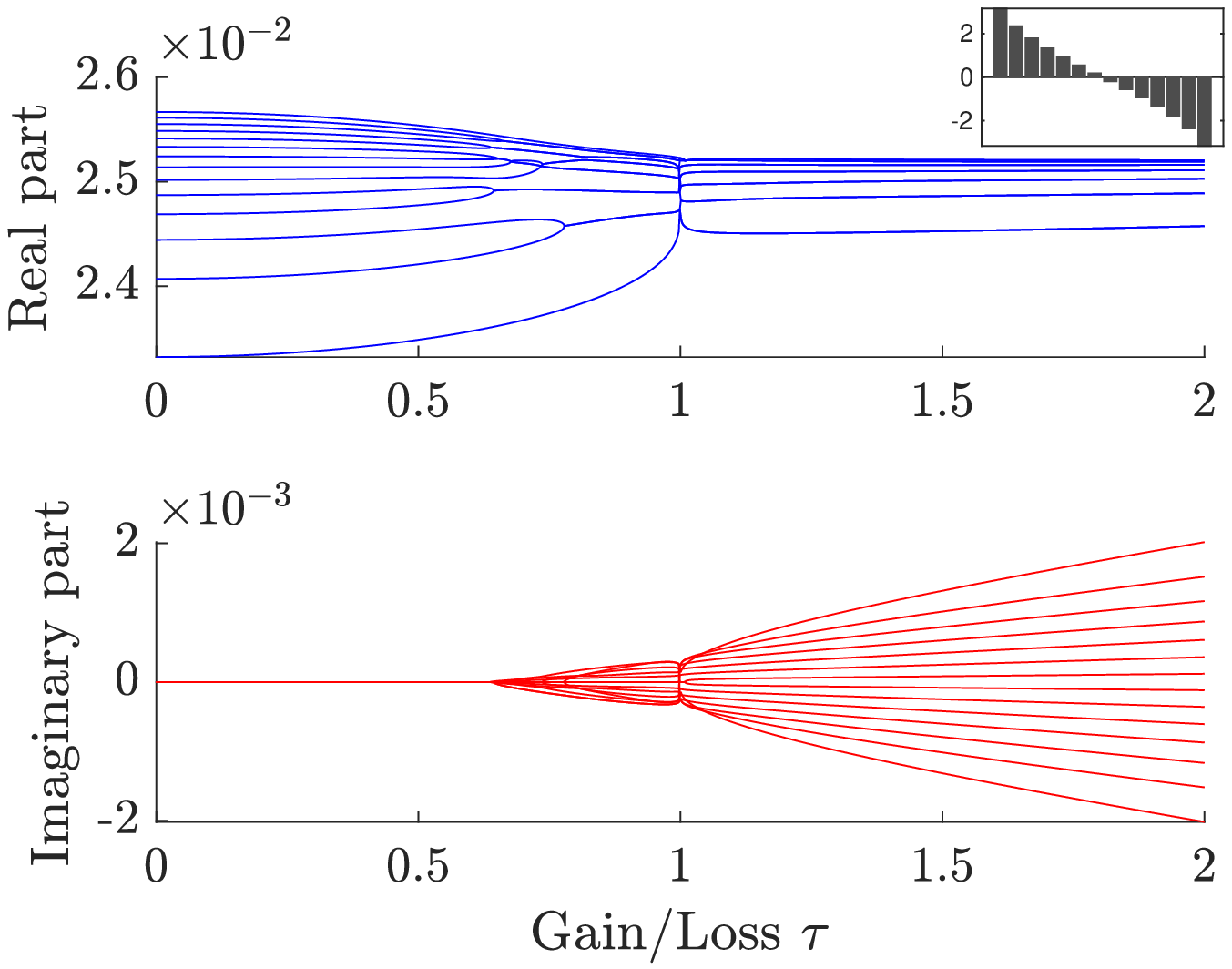}
			\caption{$N=14$.} \label{fig:N14}
		\end{subfigure}	
		\caption{Evolution of the resonant frequencies, to leading order, as gain and loss is introduced. Here, the imaginary parts $b_i$, as defined in \Cref{sec:higher}, are rescaled by $\tau \in [0,2]$, where $\tau = 1$ corresponds to an $N$\textsuperscript{th} order exceptional point. The inset plots show the relative distribution of the imaginary parts of the material parameters on each resonator.}
		\label{fig:diluteN}		
	\end{center}
\end{figure}
\begin{figure}
	\begin{center}
		\begin{subfigure}[b]{0.45\linewidth}
			\includegraphics[height=5.0cm]{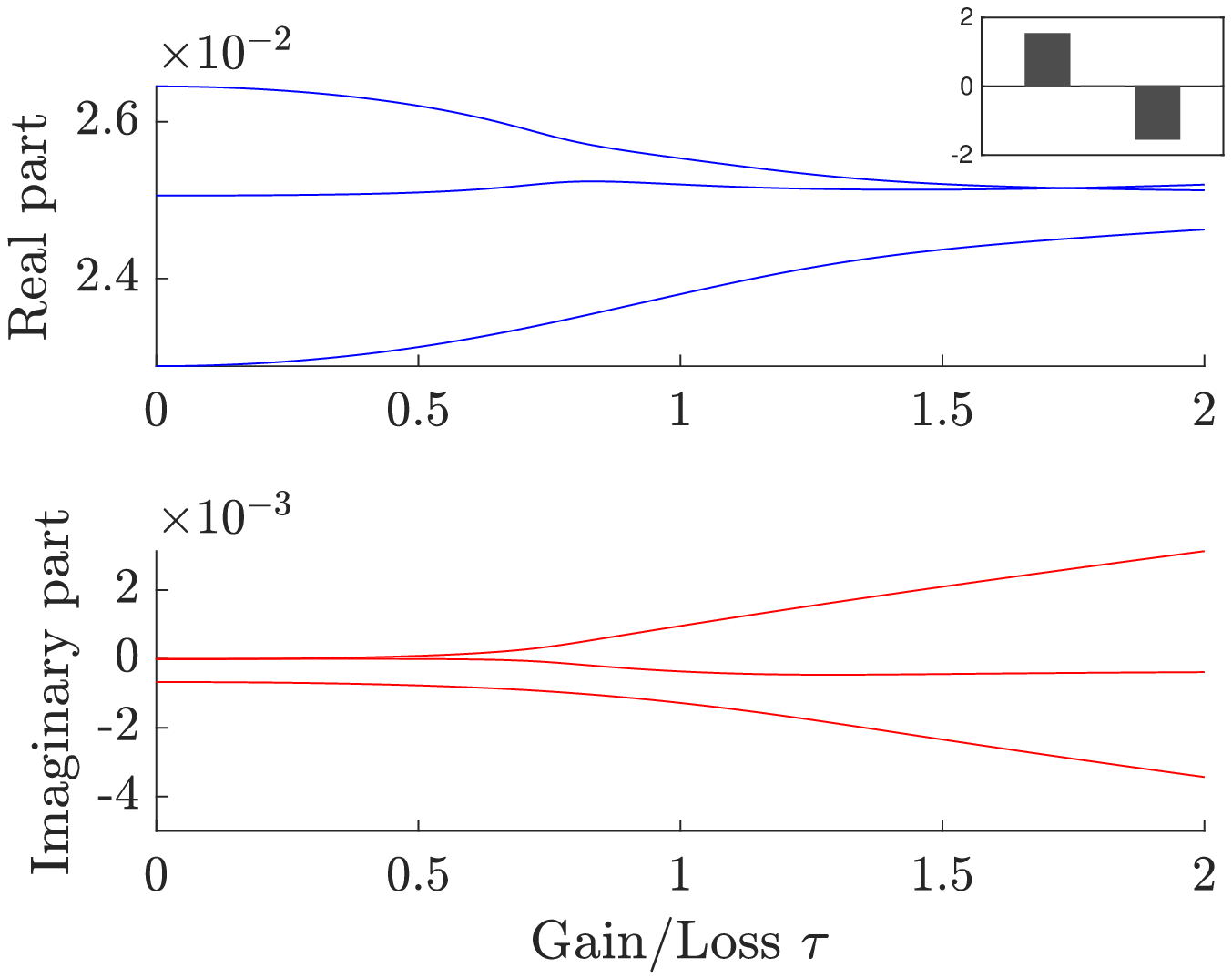}
			\caption{$N=3$} \label{fig:N3full}
		\end{subfigure}			
		\hspace{10pt}
		\begin{subfigure}[b]{0.45\linewidth}
			\includegraphics[height=5.0cm]{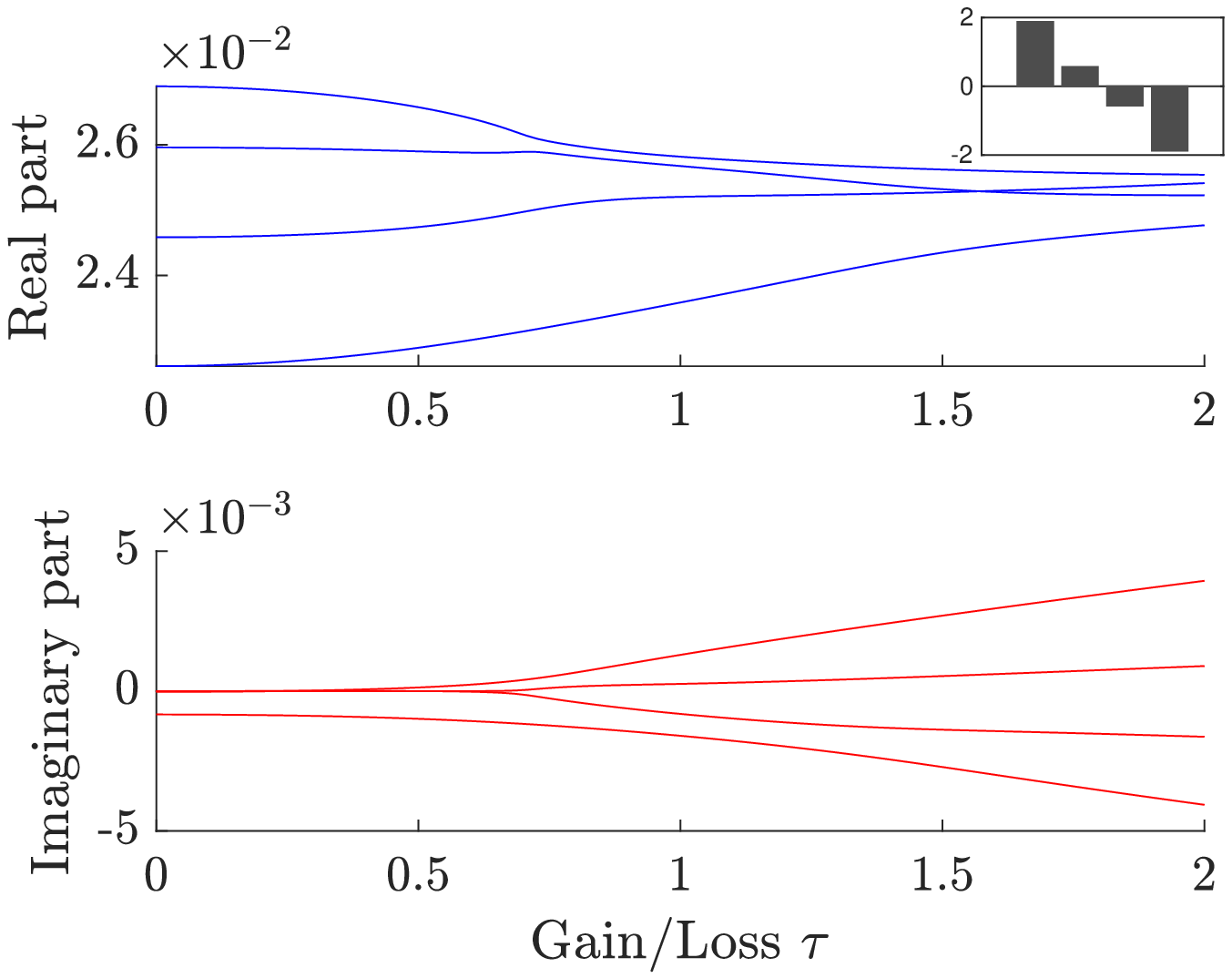}
			\caption{$N=4$.} \label{fig:N4full}
		\end{subfigure}
		\caption{Evolution of the resonant frequencies as gain and loss is introduced, computed for the full differential system without asymptotic approximation. The imaginary parts $b_i$, as defined in \Cref{sec:higher}, are rescaled by $\tau \in [0,2]$, where $\tau = 1$ corresponds to an $N$\textsuperscript{th} order asymptotic exceptional point.}
		\label{fig:fullN}		
	\end{center}
\end{figure}

The exceptional points demonstrated in \Cref{fig:diluteN} exhibit linearly growing gain/loss towards the edges of the resonator structure, which corresponds to one of the four solutions (\Cref{fig:sol4a}) from the case of four resonators. The other solutions from \Cref{fig:sols4} also have higher-order analogues. Another family of solutions, this time with alternating gain/loss (corresponding to \Cref{fig:sol4d}), is shown in \Cref{fig:Ndistr}. We emphasize that as $N$ increases, the number of different gain/loss distributions producing exceptional points vastly increases. While the realizations of these exceptional points involve matching a large number of parameters, the large number of different solutions suggests the possibility of reducing the dimensionality of the parameter space.

\begin{figure}
	\begin{center}
		\begin{subfigure}[b]{0.3\linewidth}
			\includegraphics[width=0.9\linewidth]{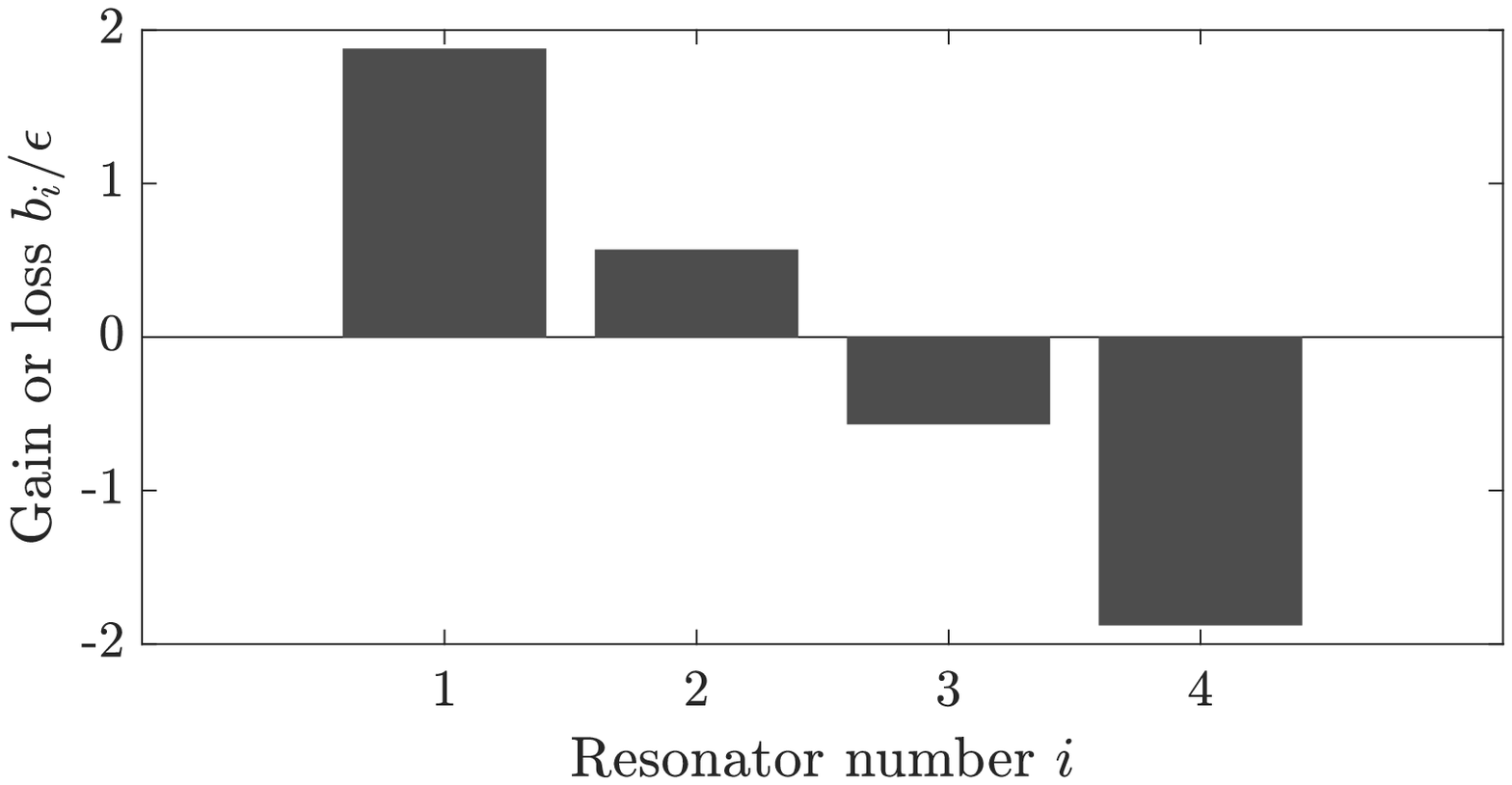}
			\caption{$N=4$} \label{fig:N4distr2}
		\end{subfigure}			
		\hspace{10pt}
		\begin{subfigure}[b]{0.3\linewidth}
			\includegraphics[width=0.9\linewidth]{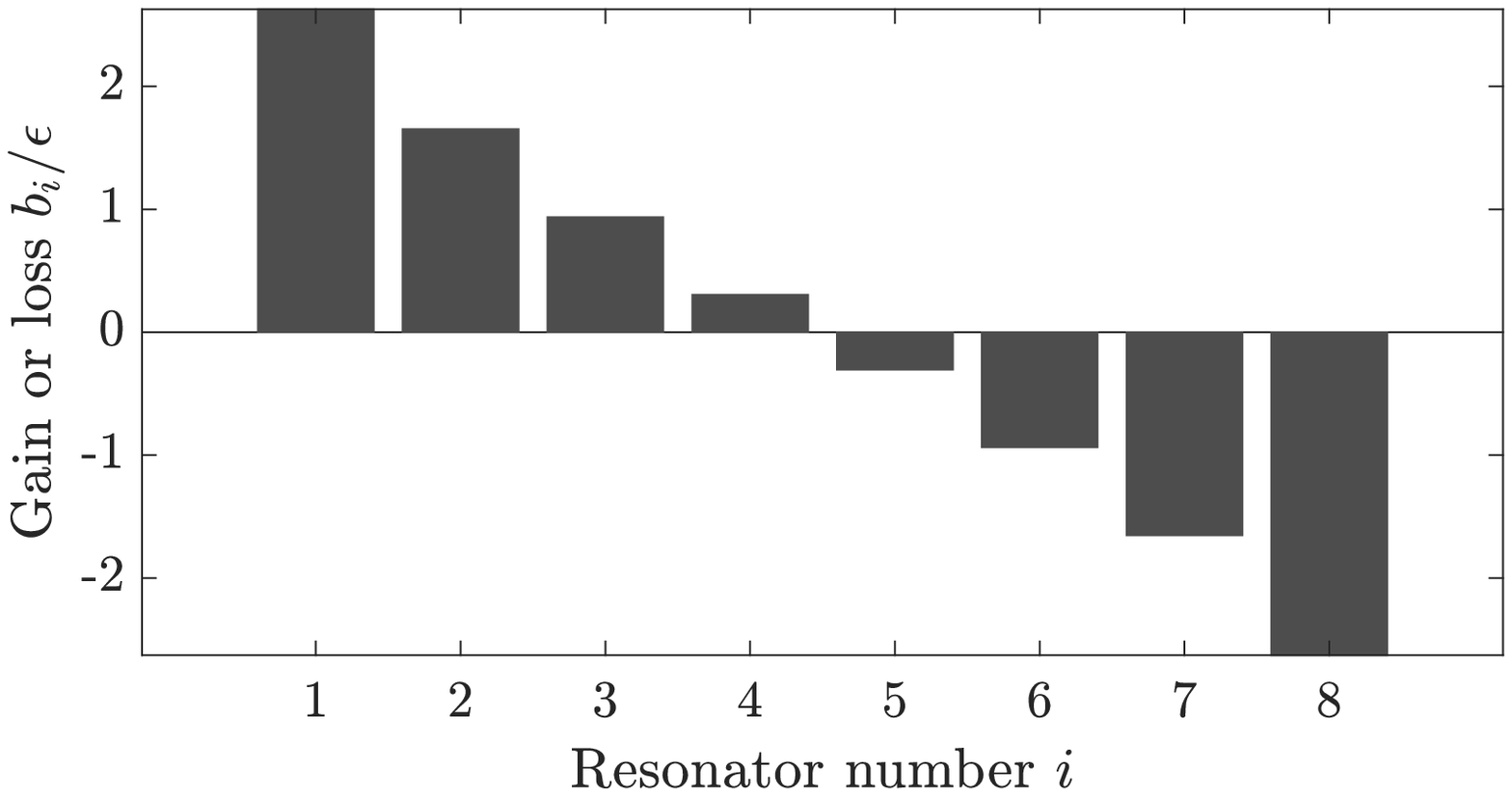}
			\caption{$N=8$.} \label{fig:N8distr2}
		\end{subfigure}
		\begin{subfigure}[b]{0.3\linewidth}
			\includegraphics[width=0.9\linewidth]{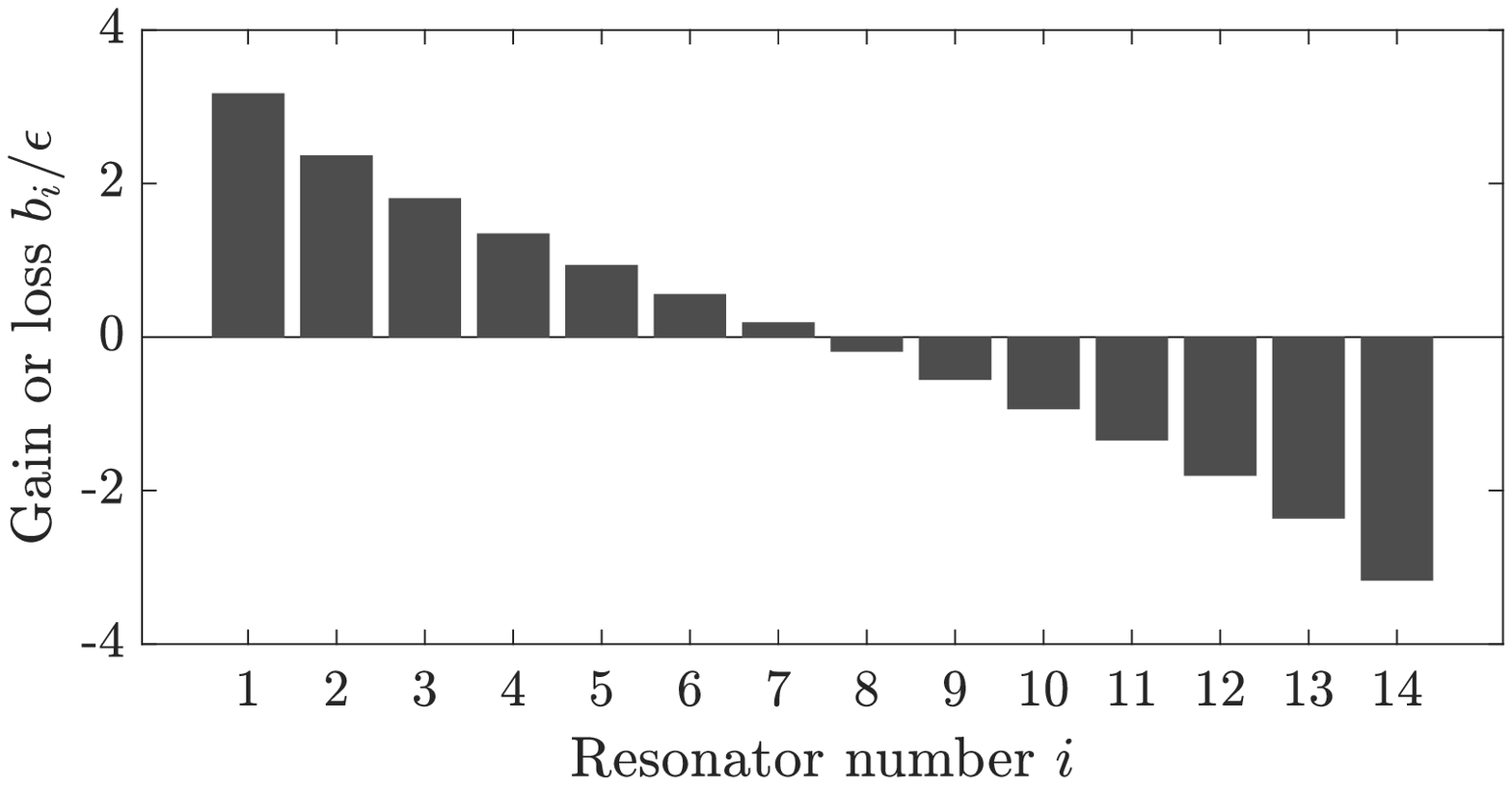}
			\caption{$N=14$} \label{fig:N14distr2}
		\end{subfigure}	
		\begin{subfigure}[b]{0.3\linewidth}
			\includegraphics[width=0.9\linewidth]{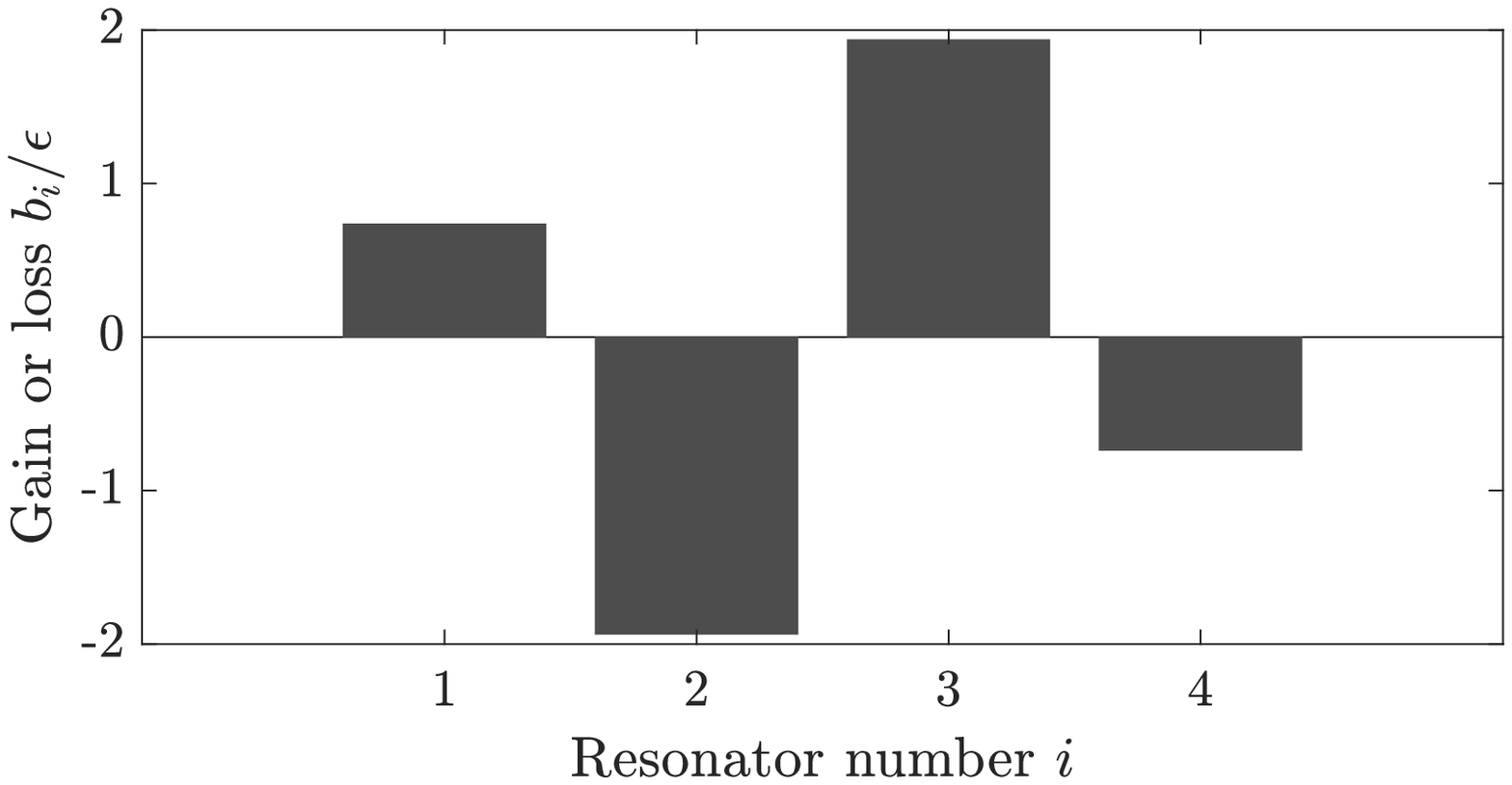}
			\caption{$N=4$} \label{fig:N4distr}
		\end{subfigure}			
		\hspace{10pt}
		\begin{subfigure}[b]{0.3\linewidth}
			\includegraphics[width=0.9\linewidth]{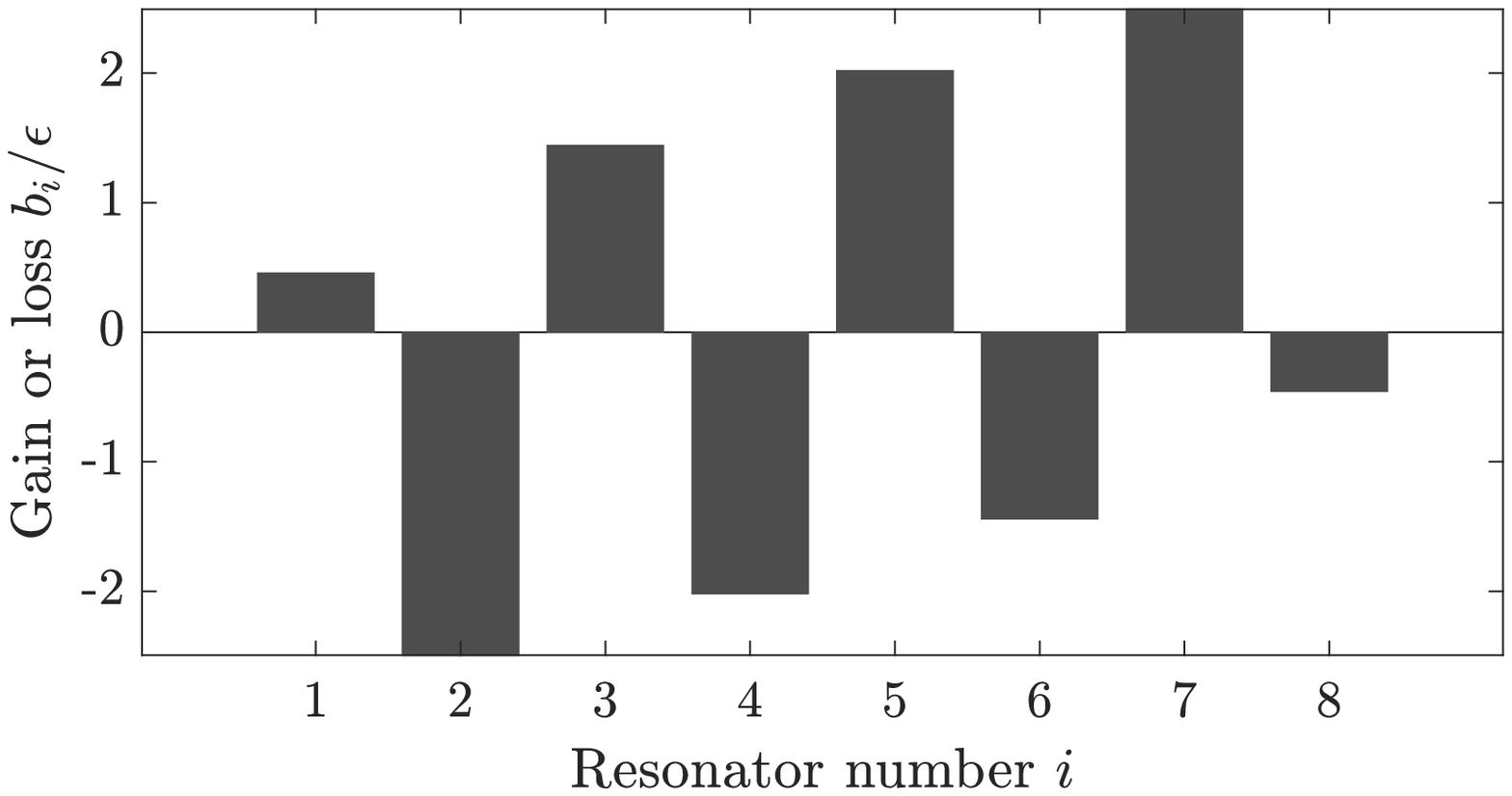}
			\caption{$N=8$.} \label{fig:N8distr}
		\end{subfigure}
		\begin{subfigure}[b]{0.3\linewidth}
			\includegraphics[width=0.9\linewidth]{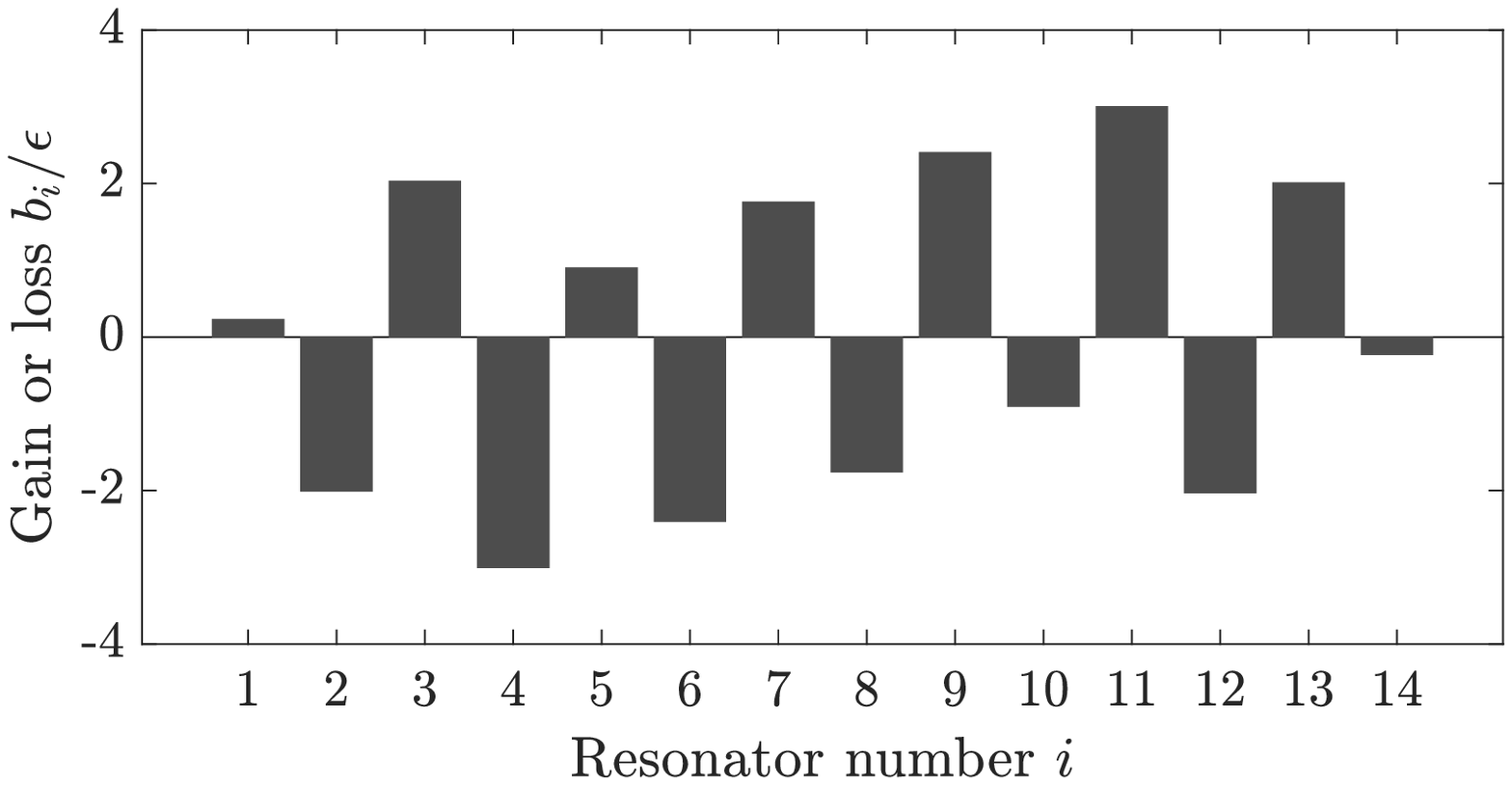}
			\caption{$N=14$} \label{fig:N14distr}
		\end{subfigure}	
		\caption{Exceptional points of different orders can be found with the same symmetries. Here, two examples of symmetric families of gain/loss distributions corresponding to exceptional points of different orders are shown. \Cref{fig:N4distr2,fig:N8distr2,fig:N14distr2} (which are the same distributions as in \Cref{fig:diluteN}) show roughly linearly growing distributions, while \Cref{fig:N4distr,fig:N8distr,fig:N14distr} show alternating distributions.
		}
		\label{fig:Ndistr}		
	\end{center}
\end{figure}

\begin{figure}
	\begin{center}
		\begin{subfigure}{0.6\linewidth}
			\includegraphics[height=5.0cm]{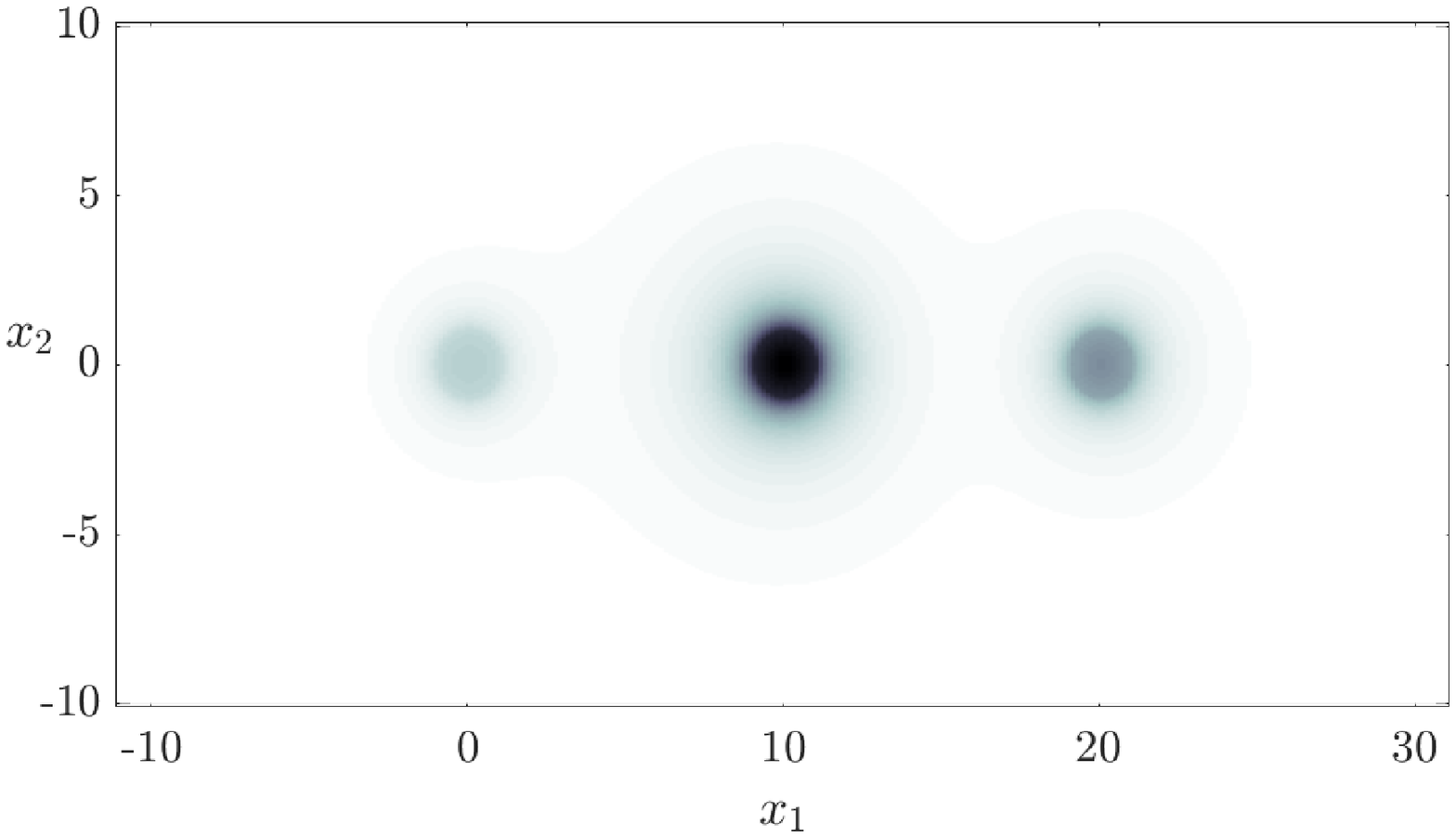}
			\caption{$|\varphi_N^2|$ for $N=3$ resonators.} \label{fig:N3pos}
	\end{subfigure}
	\vspace{15pt}
	
	\begin{subfigure}[b]{0.49\linewidth}
		\centering
		\includegraphics[width=\linewidth]{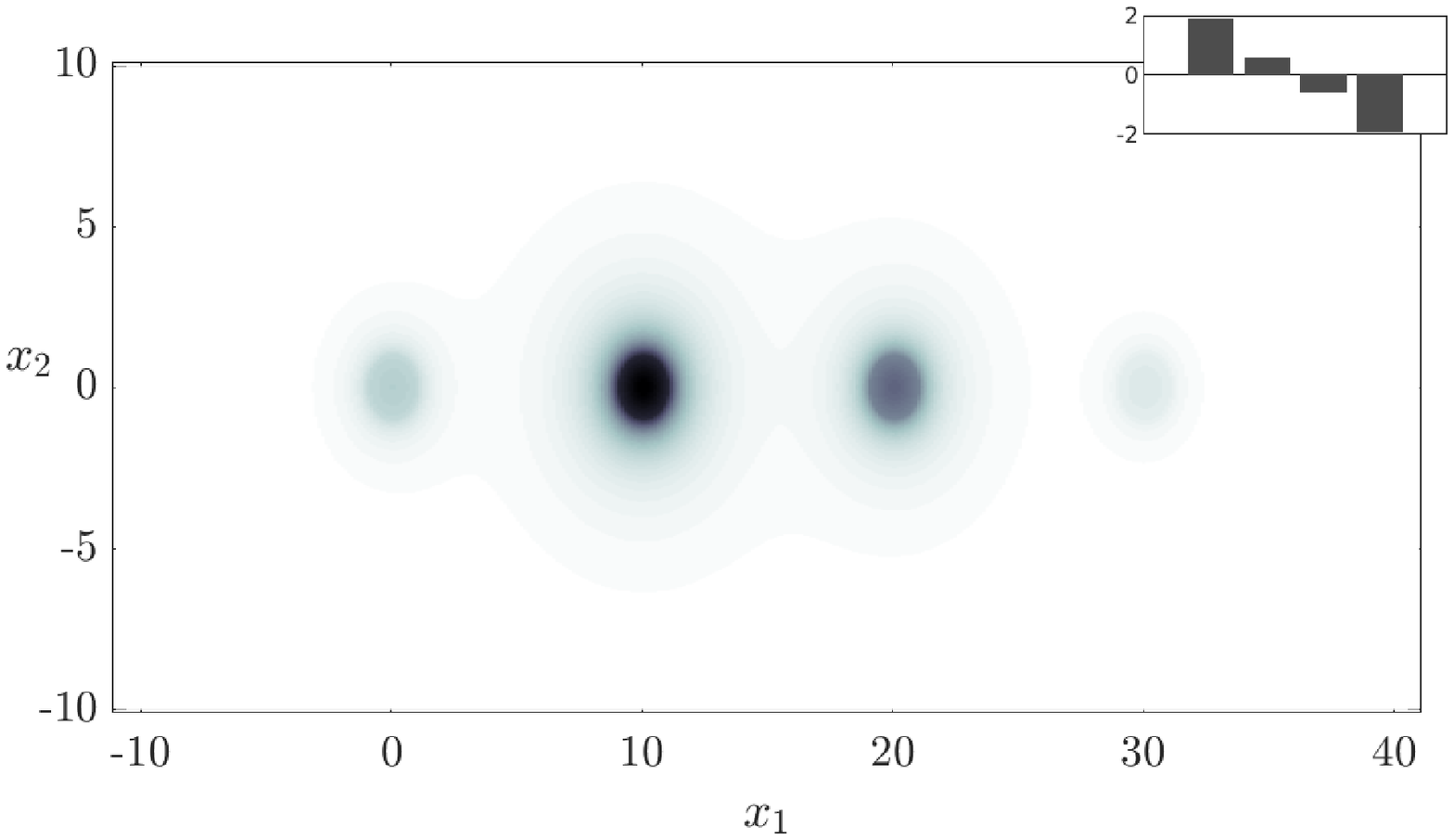}
		\captionsetup{type=figure}
	\end{subfigure}
	\begin{subfigure}[b]{0.49\linewidth}
		\centering
		\includegraphics[width=\linewidth]{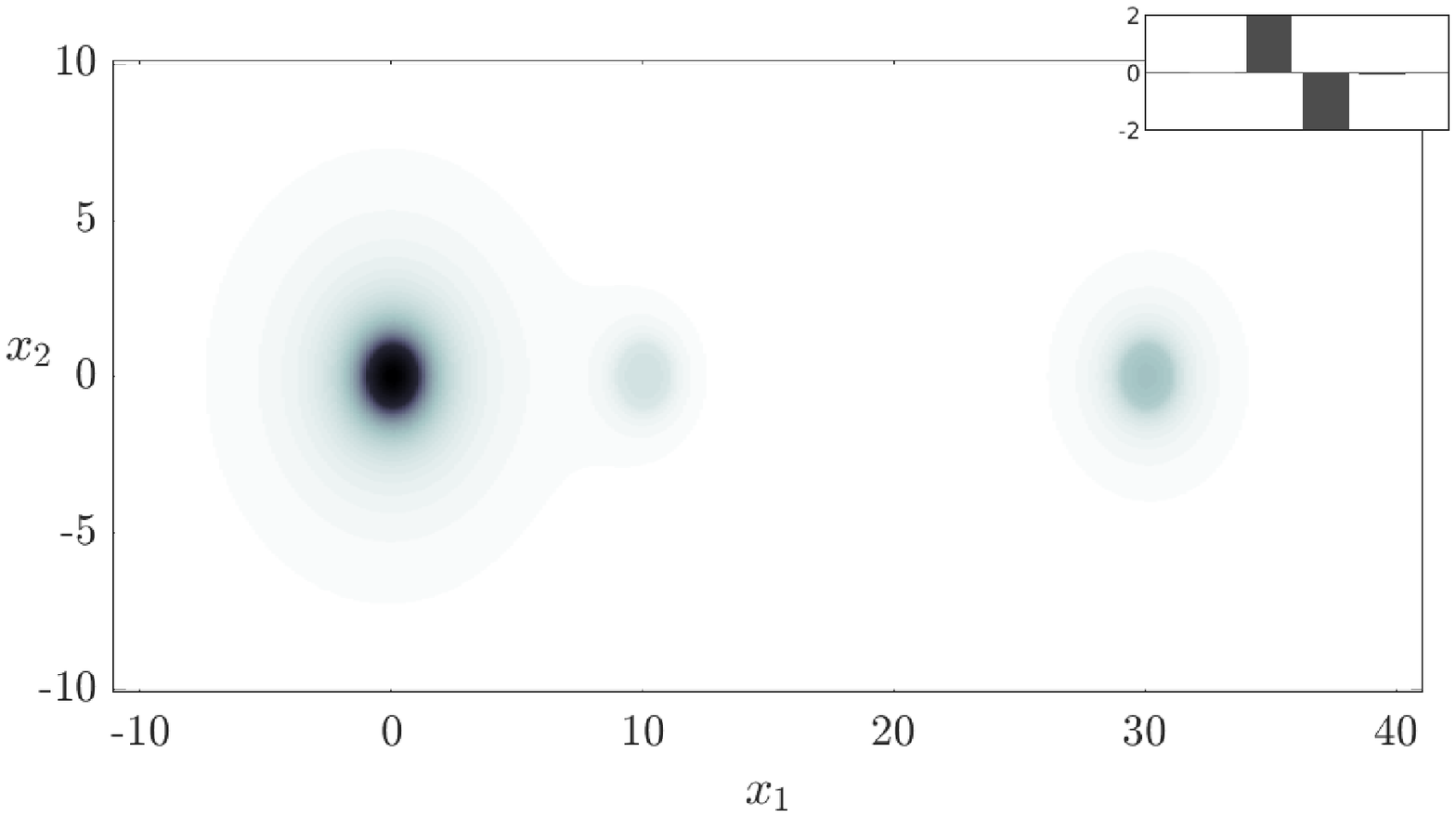}
		\captionsetup{type=figure}
	\end{subfigure}
	\vspace{5pt}
	
	\begin{subfigure}[b]{0.49\linewidth}
		\centering
		\includegraphics[width=\linewidth]{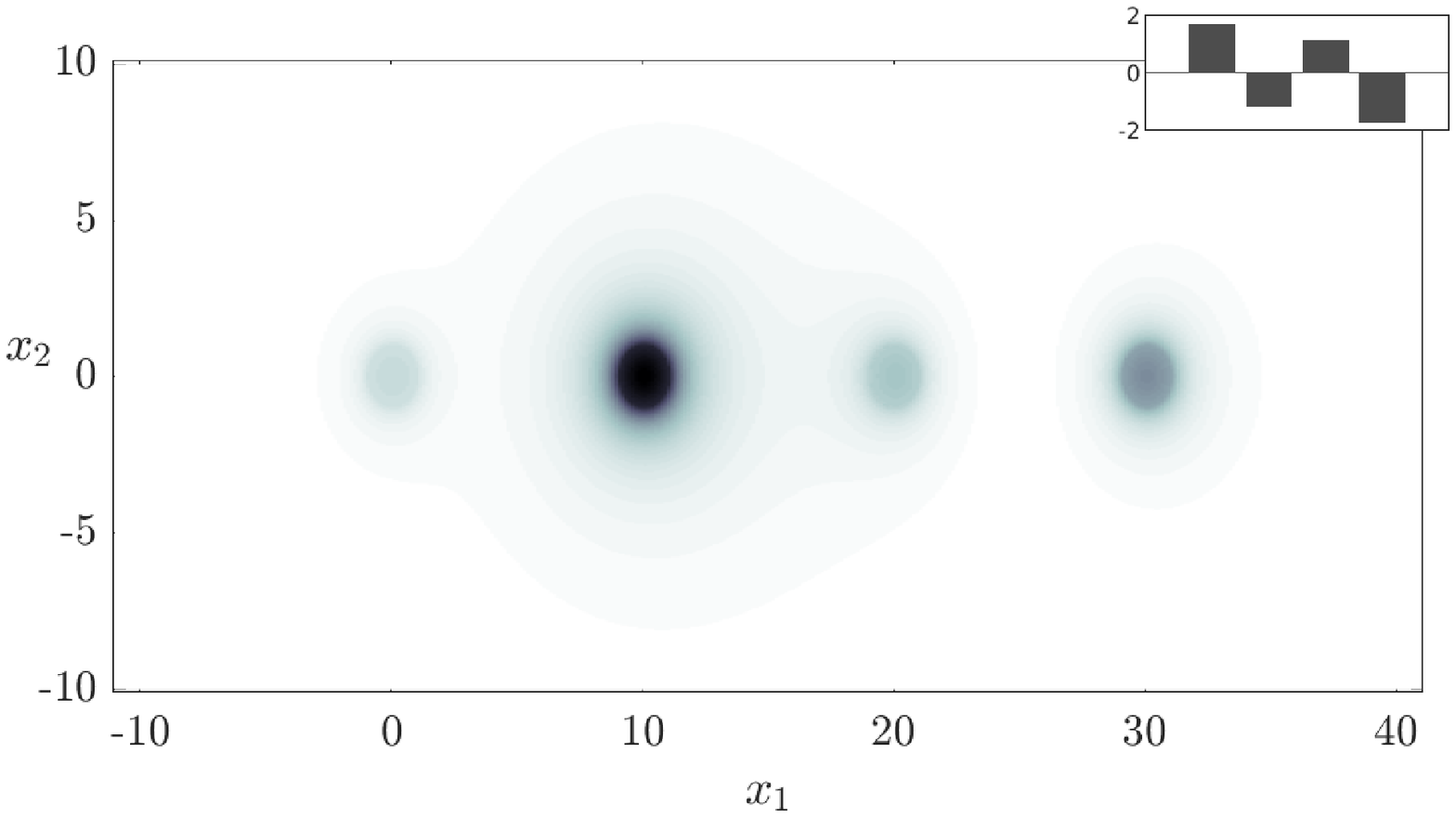}
		\captionsetup{type=figure}		
	\end{subfigure}
	\begin{subfigure}[b]{0.49\linewidth}
		\centering
		\includegraphics[width=\linewidth]{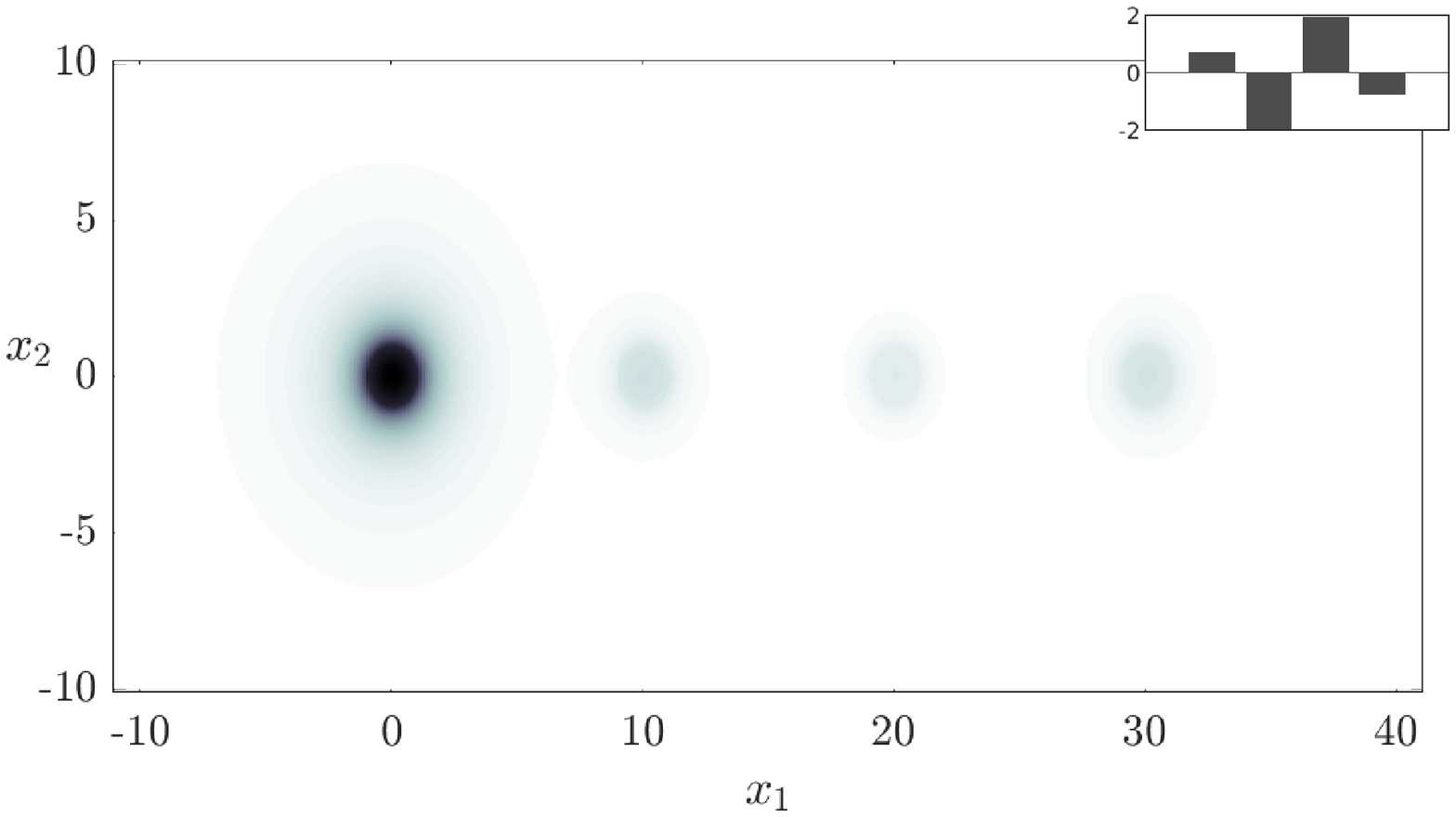}
		\captionsetup{type=figure}
	\end{subfigure}
\vspace{5pt}

\begin{subfigure}{0.65\linewidth}
	\centering
	\caption{$|\varphi_N^2|$ for each of the four asymptotic exceptional points in the case of $N=4$ resonators. The distribution of the imaginary parts on each resonator is shown inset, corresponding to the values depicted in \Cref{fig:sols4}.} \label{fig:N4pos}
	\includegraphics[width=0.7\linewidth]{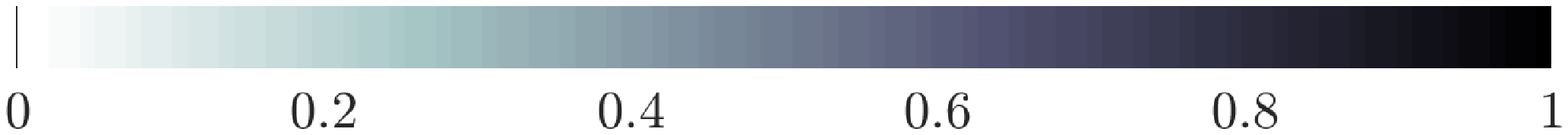}
\end{subfigure}
	\caption{In order to maximise the enhanced perturbation due to the introduction of a small particle, the position of the particle relative to the array of resonators is an important consideration. The analysis of \Cref{sec:sensing} implies that one should seek to place the particle at the point where $|\varphi_N^2|$ is maximized.}
	\label{fig:position}
	\end{center}
\end{figure}

\subsection{Enhanced sensing} \label{sec:sensing_position}

Recall that an array with an exceptional point of order $N$ has powerful applications in enhanced sensing since small perturbations typically lead to eigenfrequency shifts with a $1/N$ exponent. This was examined for the particular case of sensing the presence of a small particle in \Cref{sec:sensing}. It was shown that one of the eigenfrequencies would experience a shift proportional to $\eta_z^{1/N} |\Omega|^{1/N}$, where $|\Omega|$ is the volume of the small particle and $\eta_z$ depends on the particle's position. For small $|\Omega|$, the operator $T$ is close to the identity so $\eta_z$ is approximately proportional to $\varphi_N^2(z)$. We can thus explore the optimal position to place the small particle by plotting $|\varphi_N^2|$ and looking for its maximum. This is shown for each of the exceptional points in case of three and four resonators in \Cref{fig:position}.

The profiles presented in \Cref{fig:position} are simulations of the full differential system using the multipole method \cite[Appendix A]{ammari2020topological}. This system exhibits only \emph{asymptotic} exceptional points, which explains the unexpected symmetry. The asymptotic parameter values derived in \Cref{sec:third,sec:fourth-order} were used in these simulations.

As is typically the case with structures composed of many subwavelength resonators (see \emph{e.g.} \cite{ammari2019fully, ammari2020topological}), the maxima of the resonant modes occur on the resonators themselves. Thus, for optimal enhanced sensing of a small particle, the particle should be placed close to one of the resonators. The choice of which resonator is ideal varies depending on the precise configuration of the exceptional point in question, as demonstrated by \Cref{fig:N4pos}.


\section{Concluding remarks}
In this work, we have demonstrated the possibility of using subwavelength resonators for sensitivity enhancement. There are two key steps in the argument. First, we have shown that enhanced sensing occurs at exceptional points, and proven how the sensitivity is enhanced by increasing the order of the exceptional point. Then, we have demonstrated that high-order asymptotic exceptional points indeed occur in systems of subwavelength resonators. We have rigorously proven that a third-order asymptotic exceptional point exists, and have numerically demonstrated a plethora of configurations giving exceptional points of higher-orders.

\section*{Data availability}

The code used in this study is available at \href{https://github.com/davies-b/highEPs}{\texttt{https://github.com/davies-b/highEPs}}. 

\bibliographystyle{abbrv}
\bibliography{exceptional_higher}{}

\appendix
\section{Second-order exceptional points} \label{sec:2res}
The approach used in this work can also be used to find an exceptional point in a $\mathcal{PT}$-symmetric pair of resonators. This structure was previously studied without the assumption of diluteness in \cite{ammari2020exceptional}. The two resonators have material parameters given by
$$v_1^2 \delta_1 := \delta a(1 + \iu{b}), \qquad v_2^2 \delta_2 := \delta a(1 - \iu{b}),$$
In this case we wish to find an exceptional point of the matrix 
\begin{equation}
C^v_d = \begin{pmatrix}
\ds 1+\iu b   &\ds -(1+\iu b) \epsilon 
\\
-(1-\iu b) \epsilon &\ds 1-\iu b 
\end{pmatrix}.
\end{equation}
This has characteristic polynomial given by
\begin{equation}
P(x)=x^2-2x+(1+b^2)(1-\epsilon^2),
\end{equation}
which we want to have the form $(x-\gamma)^2=x^2-2\gamma x+\gamma^2$. This can be achieved by choosing $\gamma=1$ and 
\begin{equation}
b=\frac{\epsilon}{\sqrt{1-\epsilon^2}}.
\end{equation}
We must also verify that $\dim \ker (C_d^v - I) = 1$. We have that
\begin{equation}
C_d^v - I = \begin{pmatrix} \iu \epsilon & -\epsilon \\ -\epsilon & -\iu\epsilon \end{pmatrix} +O(\epsilon^2),
\end{equation}
from which we can see that $\ker (C_d^v - I)$ is spanned by $(-\iu,1)^\mathrm{T}$, at leading order in $\epsilon$, and is one dimensional.

\end{document}